\documentclass[article,11pt]{amsart}

\usepackage{amsthm}
\usepackage{amssymb}
\usepackage{amstext}
\usepackage{amsmath}
\usepackage{dsfont}
\usepackage{url}
\usepackage[colorlinks]{hyperref}
\hypersetup{
linkcolor=blue,
citecolor=blue,
}
\usepackage[margin = 40pt,font=small,labelfont=bf]{caption}
\usepackage{arydshln}
\usepackage{pdfsync}
\usepackage[mac]{inputenc}
\usepackage[T1]{fontenc}
\usepackage{color}
\usepackage[table]{xcolor}
\usepackage{algorithm}
\usepackage{algorithmic}
\usepackage{enumerate}
\usepackage{bbm}
\usepackage{ae,aecompl}
\usepackage{pdfpages}
\usepackage{epstopdf}
\usepackage{graphicx}
\usepackage{epsfig}
\usepackage{subfigure}
\usepackage{stmaryrd}

\numberwithin{equation}{section}
\theoremstyle{plain}
\newtheorem{thm}{Theorem}[section]
\newtheorem{rem}{Remark}[section]
\newtheorem{ex}{Example}[section]
\newtheorem{prop}{Proposition}[section]

\newtheorem{lem}{Lemma}[section]
\newtheorem{definition}{Definition}[section]
\newtheorem{hyp}{Assumption}[section]

\def\build#1_#2^#3{\mathrel{\mathop{\kern 0pt#1}\limits_{#2}^{#3}}}

\def\videbox{\mathbin{\vbox{\hrule\hbox{\vrule height1.4ex \kern.6em\vrule height1.4ex}\hrule}}}

\newcommand{\RR}{{\mathbb R}}

\renewcommand{\P}{{\mathbb P}}

\newcommand{\Bd}{\mathbb{B}(0,1)\backslash \{0\}}
\newcommand{\EE}{\ensuremath{\mathbb E}}

\newcommand{\XX}{\ensuremath{\mathcal X}}

\newcommand{\ee}{\ensuremath{\varepsilon}}

\newcommand*{\supl}{\operatornamewithlimits{sup}\limits}

\newcommand*{\argsupl}{\operatornamewithlimits{argsup}\limits}

\def\argmin{\mathop{\rm arg \; min}\limits}%

\newcommand{\thefont}[2]{\fontsize{#1}{#2}\fontshape{n}\selectfont}
\newcommand{\1}{\rlap{\thefont{10pt}{12pt}1}\kern.16em\rlap{\thefont{11pt}{13.2pt}1}\kern.4em}

\begin{document}

\title[]{Monge-Kantorovich superquantiles and expected shortfalls with applications to multivariate risk
measurements\vspace{1ex}}

\author{Bernard Bercu, J\'{e}r\'{e}mie Bigot and Gauthier Thurin}
\dedicatory{\normalsize Universit\'e de Bordeaux \\
Institut de Math\'ematiques de Bordeaux et CNRS  UMR 5251, \\ 351 Cours de la lib\'eration, 33400 Talence cedex, France}
\thanks{The authors gratefully acknowledge financial support from the Agence Nationale de la Recherche  (MaSDOL grant ANR-19-CE23-0017).}

\maketitle

\thispagestyle{empty}

\vspace{-2ex}

\begin{abstract}
We propose center-outward superquantile and expected shortfall functions, with applications to multivariate risk measurements, extending the standard notion of value at risk and conditional value at risk from the real line to $\RR^d$. 
Our new concepts are built upon the recent definition of Monge-Kantorovich quantiles based on the theory of optimal transport, and they provide a natural way to characterize multivariate tail probabilities and central areas of point clouds. 
They preserve the univariate interpretation of a typical observation that lies beyond or ahead a quantile, but in a meaningful multivariate way. 
We show that they characterize random vectors and their convergence in distribution, which underlines their importance. Our new concepts are illustrated on both simulated and real datasets. 
\end{abstract}

\vspace{0.5cm}

\noindent \emph{Keywords:} Monge-Kantorovich quantiles, center-outward quantiles, tails of a multivariate distribution, conditional value at risk, expected shortfall.\\

\noindent\emph{AMS classifications:} 62H05, 62P99, 60F05.

\section{Introduction}
\subsection{Superquantile, expected shortfall}
Modeling the dependency between the components of a random vector is at the core of multivariate statistics. To that end, one way to proceed is to characterize the multivariate probability tails. For distributions supported on the real line, this is often tackled with the use of superquantiles or expected shortfalls, that complement
 the information given by the quantiles. 
Let $X $ be an integrable absolutely continuous random variable with cumulative distribution function $F$. For all $\alpha \in ]0,1[$, the quantile $Q(\alpha)$ of level $\alpha$ is given by
$$
Q(\alpha) = \inf\{ x : F(x) \geq \alpha \},
$$
whereas the superquantile $S(\alpha)$ and expected shortfall $E(\alpha)$  are defined by 
\begin{equation}\label{superQdim1}
S(\alpha) = \EE [X \big\vert X \ge Q(\alpha)] = \frac{\EE [X \mathds{1}_{ X \ge Q(\alpha)}] }{\mathbb{P} (X \ge Q(\alpha))} = \frac{1 }{1-\alpha} \EE [X \mathds{1}_{ X \ge Q(\alpha)}],
\end{equation}
and
\begin{equation}\label{ESdim1}
E(\alpha) = \EE [X \big\vert X \le Q(\alpha)] = \frac{\EE [X \mathds{1}_{ X \le Q(\alpha)}] }{\mathbb{P} (X \le Q(\alpha))} = \frac{1}{\alpha} \EE [X \mathds{1}_{ X \le Q(\alpha)}].
\end{equation}

\noindent As illustrated in Figure \ref{fig:plotintro}, $S(\alpha)$ focuses on the upper-tail while $E(\alpha)$ targets {  the lower-tail.
We emphasize that} using the terms of superquantile and expected shortfall is a subjective consideration 
taken from \cite{acerbi2001ES,Rockafellar2013}.
Most of the time, one does not consider the upper and the lower tails together, so that a single name is required, up to considering  the distribution of $-X$. 
In this vein, depending on the application, the expected shortfall may refer to the same as the Conditional-Value-at-Risk, Conditional-Tail-Expectation, see $e.g.$ \cite{acerbi2002coherence}, or even the superquantile, that aims to be a neutral alternative name in statistics \cite{Rockafellar2013}. 

\begin{figure}[htbp]
\centering
\includegraphics[width=3.2in,height=1.6in]{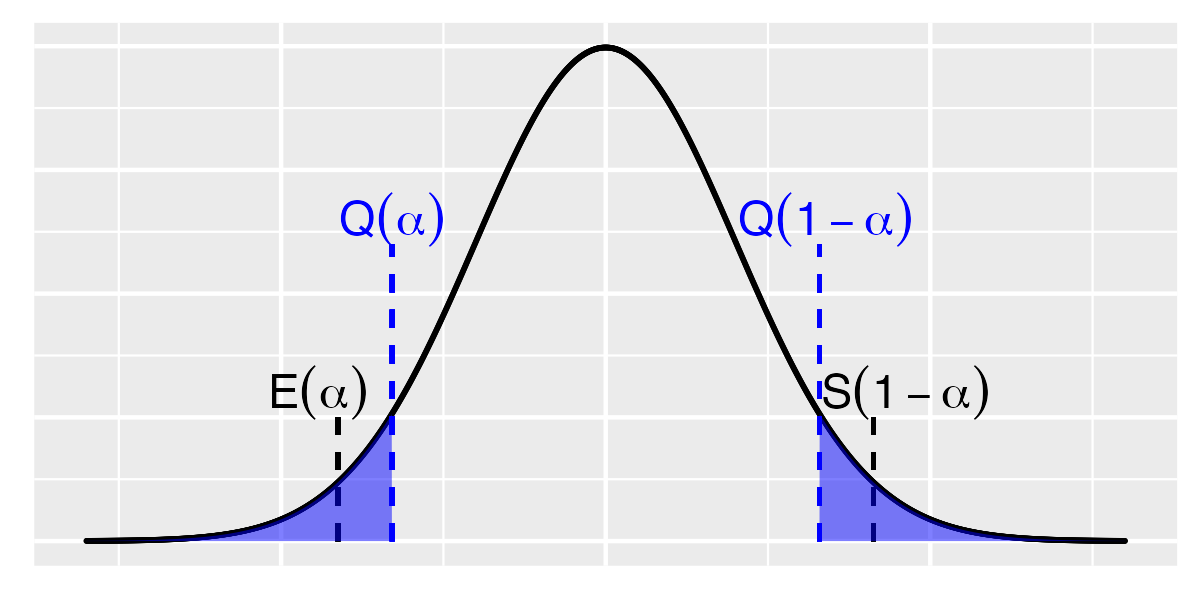}
\vspace{-2ex}
\caption{ Illustration of the notions of superquantile  $S(\alpha)$ and expected shortfall $E(\alpha)$ for an univariate Gaussian distribution.}
\label{fig:plotintro}
\end{figure}

The main contribution of the present paper is to extend \eqref{superQdim1} 
and \eqref{ESdim1} towards a notion of multivariate superquantile and expected shortfall. 
As part of the difficulty, both the mathematical meanings of ``ahead'', ``beyond'' and ``typical'' do not adapt canonically in $\RR^d$. 
We argue that sufficient notions are provided by the Monge-Kantorovich (MK) quantiles, ranks and signs, introduced in \cite{chernozhukov2015mongekantorovich,Hallin-AOS_2021}. 
In particular, the traditional left-to-right ordering is replaced in our approach by a center-outward one 
that is more intuitive for a point cloud \cite{chernozhukov2015mongekantorovich}.
Hence, the two subsets of observations that we are interested in are located at the outward or near the mean value, which requires to adapt the concepts of \eqref{superQdim1} and \eqref{ESdim1} in $\RR^d$. 
It is well-known from a simple change of variables in $\RR$ that $S$ and $E$ average observations beyond and ahead the quantile of level $\alpha$, in the sense that 
\begin{equation}\label{EandS_intro_dim1}
S(\alpha) = \frac{1}{1-\alpha} \int_\alpha^1 Q(t) dt
\hspace{1cm}
\text{and}
\hspace{1cm}
E(\alpha) = \frac{1}{\alpha} \int_0^\alpha Q(t) dt.
\end{equation}

In Section \ref{superquant}, our definitions generalize the formulation \eqref{EandS_intro_dim1}. 
If $Q$ stands for the multivariate MK quantile function instead of the classical univariate one, center-outward superquantile and expected shortfall functions are defined, for any 
$u$ in the unit ball $\mathbb{B}(0,1)\backslash \{ 0\}$, by
\begin{equation*}\label{EandS_intro}
S(u) = \frac{1}{1-\Vert u\Vert} \int_{\Vert u\Vert}^1 Q\Big(t \frac{u}{\Vert u \Vert}\Big) dt
\hspace{1cm}
\text{and}
\hspace{1cm}
E(u) = \frac{1}{\Vert u\Vert} \int_0^{\Vert u\Vert} Q\Big(t\frac{u}{\Vert u \Vert}\Big) dt.
\end{equation*}

MK quantiles have already led to many applications, among which lie statistical tests \cite{ghosal2021multivariate,HallinTests2022,Ziang2022,shi2021,shi2022distribution}, regression \cite{Carlier:2022wq,delBarrio2022}, risk measurement \cite{beirlant2019centeroutward}, or Lorenz maps \cite{Fan2022,Hallin_mordant_2022}. We also refer to the recent review \cite{HallinReport2021} 
on the concept of MK quantiles.
Importantly, the univariate quantile and related functions are deeply rooted in risk analysis.
On the one hand, risk measures which are both coherent and regular can be characterized by integrated quantile functions \cite{Guschin_2017}. 
On the other hand, given a level $\alpha$, fundamental risk measures are given by $Q(\alpha)$ and $S(\alpha)$, called  Value-at-Risk (VaR) and  Conditional-Value-at-Risk (CVaR),  respectively. 
As a matter of fact, a natural main contribution of the present paper is to provide meaningful multivariate extensions of VaR and CVaR.

\subsection{Background on multivariate risk measurement}
MK quantiles and the related concepts have already been applied to multivariate risk measurement in \cite{beirlant2019centeroutward}, and, in some sense, in the previous works \cite{Ekeland_2010,Henry2021}. The theory in \cite{Ekeland_2010} states ideal theoretical properties for coherent regular risk measures, while the maximal correlation risk measure of \cite{beirlant2019centeroutward} furnishes a real-valued risk measure with these properties. 
This constitutes, to the best of our knowledge, the short literature on risk measurement based on the MK quantile function. 
In this work, we argue that an adequate procedure of multivariate risk measurement shall account for all the information on the tails, both in terms of direction and spreadness. To answer this issue, vector-valued risk measures are natural candidates.{  
There} also exist several extensions of VaR or CVaR to the multivariate setting, including \cite{Armaut2023,huameng2022,COUSIN201332,DiBernardino2014,Goegebeur2023,Janet2004,Prekopa2012,Torres2015,Kerem2021}, but none of them is based on the theory of optimal transportation and its associated potential benefits. 
In particular, our concepts do not require any assumption on the tail behavior of the data, nor any statistical model, because MK quantiles adapt naturally to the shape of a point cloud. 
{  
On the real line,} the VaR and the CVaR have a clear interpretation: for a level $\alpha \in [0,1]$, the VaR is the worst observation encountered with probability $1-\alpha$ whereas CVaR is the average value beyond this worst observation. Such a meaningful definition is surely part of the reason for their wide use in practice. Under the name of Conditional-Tail-Expectation, the idea proposed 
in \cite{DiBernardino2014,di2013plug} preserves this interpretation, but relies on level 
sets defined from the theory of copulas. 
Specifically, the obtained quantile levels do not adapt automatically to the shape of the data. Still, this notion averages over a certain quantile level, and it returns a tail observation of the same dimension as the data. Our work is inspired by this approach, as we aim to give the same information about multivariate tails, but we use the MK quantile function, which yields, to our opinion, concepts with even better interpretability. 

\subsection{Main contributions}

Our main contributions are the definitions of a center-outward superquantile function and related risk measures, both real-valued and vector-valued. Doing so, we provide an extension to the multivariate case of the fundamental 
Value-at-Risk and Conditional-Value-at-Risk. 
Furthermore, we provide a center-outward expected shortfall function, that describes the central areas of a given point cloud. 
Our center-outward expected shortfalls and superquantiles are uniquely defined and characterize convergence in distribution, and they are closely related to the potential whose gradient gives the MK quantile function. 
A result of independent interest is also provided, giving a new family of Monge maps between known probability distributions. This may motivate changing the reference distribution under generalized gamma models. 

\subsection{Outline of the paper}
Section \ref{superquant} details  definitions and properties of our new concepts of center-outward superquantiles and expected shortfalls. It includes our main results and the crucial relation between these functions and Kantorovich potentials. 
{  An alternative class of reference measures, which differ from the standard spherical uniform distribution on the unit ball, is provided in Section \ref{class_ref}.}
The multivariate definitions of VaR and CVaR are given in Section \ref{sec:VVrisk}. Finally, we present in Section \ref{sec:numexp} a regularized version of our superquantile and expected shortfall functions, using entropically regularized optimal transport that has fundamental computional benefits to estimate the center-outward quantile function. Numerical experiments are also provided to shed some light on the benefits of our new concepts of MK superquantile, expected shortfall and multivariate VaR and CVaR for multivariate data analysis. A conclusion and some perspectives are given in Section \ref{sec:conclu}.
 
\section{Center-outward superquantiles and expected shortfalls}\label{superquant}

\subsection{Main definitions}

On the real line, the notion of superquantile and expected shortfall relies heavily on the one of quantile. It is then natural to make use of the Monge-Kantorovich (MK) quantile function and its appealing properties in order to define associated superquantile and expected shortfall functions. 
By simplicity, we shall restrict ourselves to the set of integrable probability measures over $\RR^d$, that is 
\begin{equation*}
\mathcal{P}_1(\RR^d)= \big\{ \nu : \EE_{X\sim \nu}[\Vert X \Vert] < +\infty \big\}.
\end{equation*}
Recall that
a probability measure $\nu \in \mathcal{P}_1(\RR^d)$  is the push-forward of $\mu \in \mathcal{P}_1(\RR^d)$ by $T : \RR^d \rightarrow \RR^d$ if $T(U)$ has distribution $\nu$ as soon as $U$ is distributed according to $\mu$.
This is denoted by
$T_\# \mu = \nu$. 
The following has been introduced in \cite{Hallin-AOS_2021}. 
\begin{definition}\label{defQ}
The MK quantile function of a multivariate distribution $\nu$, with respect to a reference distribution $\mu$, is a push-forward map $Q_\# \mu = \nu$ such that
there exists a convex potential $ \psi : \RR^d \rightarrow \RR$ satisfying $\nabla \psi = Q$ $\mu$-almost everywhere.
\end{definition}

It follows from McCann's theorem, \cite{McCannThm} that, as soon as $\mu$ is absolutely continuous, such a Monge map $Q$ exists and is unique.
Moreover, if $\mu$ and $\nu$ have finite moments of order two, by the result known in the literature as Brenier's theorem \cite{Brenier1991PolarFA,Cuesta1989NotesOT}, $Q$ is characterized as the solution of the following Monge problem of optimal transport,
\begin{equation}\label{MongeOT}
Q = \argmin_{T:T_\#\mu = \nu} \int_\XX \Vert u - T(u) \Vert^2 d\mu(u).
\end{equation}

Intuitively, the reference distribution $\mu$ must be chosen so that a relevant notion of quantiles can be derived from it, whereas being the gradient of a convex function $\psi$ is a generalization of monotonicity. For instance, one can choose the \textit{spherical uniform} distribution, denoted by $\mu=U_d$. It is given by the product $R \Phi$ between two independent random variables $R$ and $\Phi$, being drawn respectively from a uniform distribution on $[0,1]$ and on the unit sphere. 
Samples from $U_d$ are distributed from the origin to the outward within the unit ball, so that the balls of radius $\alpha \in [0,1]$ have probability $\alpha$ while being nested, as $\alpha$ grows. With this in mind, the hyperspheres of radius $\alpha$ are relevant quantile contours with respect to $\mu$. 
Being the gradient of a convex function, $Q$ adequately transports this \textit{center-outward} ordering towards the distribution $\nu$. 
Thus, when $\mu = U_d$, we shall refer to $Q$ as the \textit{center-outward} quantile function of $\nu$. 
This property is illustrated in Figure \ref{fig:UdBananeIntro}, where radius, in red, and circles, in blue, are transported from the unit ball to a banana-shaped distribution thanks to the mapping $Q$ obtained with the computational approach described in Section \ref{sec:numexp}. 

\begin{figure}[htbp]
\includegraphics[width=4.4in,height=2.2in]{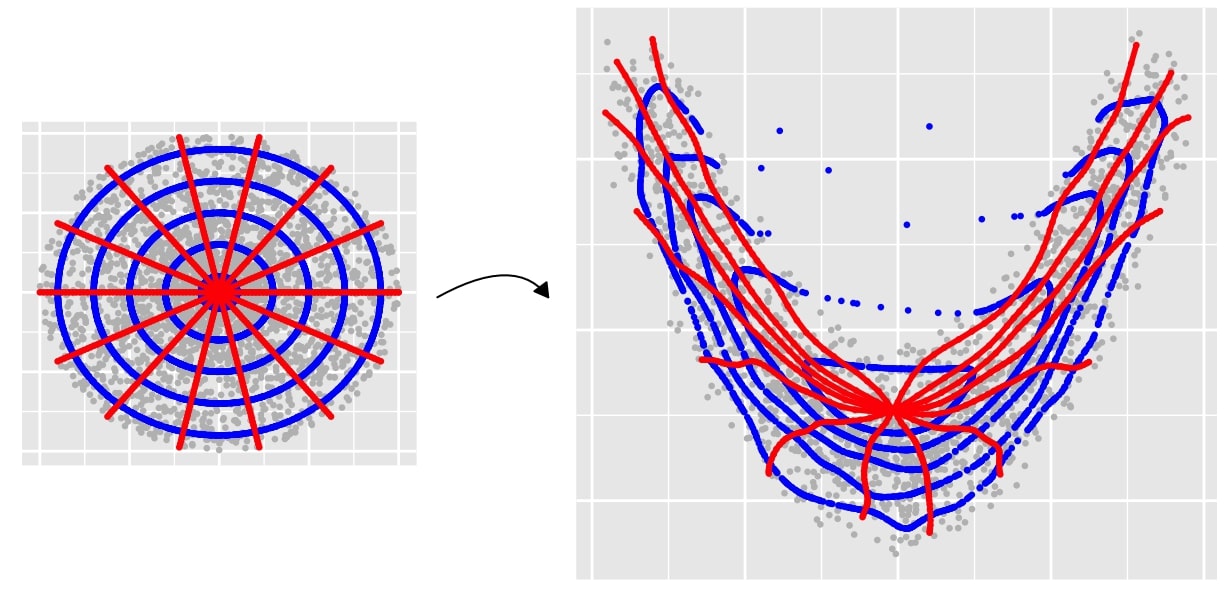}
\caption{ (Left) center-outward quantiles of the spherical uniform distribution $\mu = U_d$, and (right) center-outward quantiles of a discrete distribution $\nu$ obtained by $Q_\# \mu = \nu$.}
\label{fig:UdBananeIntro}
\end{figure} 
\ \vspace{-1ex} \par
This relevant ordering clearly catches the geometry of the support of the target distribution $\nu$, and it comes with quantile {  regions} indexed by a probability level $\alpha \in [0,1]$, by use of the change of variables formula for push-forward maps. 
More details are given in \cite{chernozhukov2015mongekantorovich,Hallin-AOS_2021}. 
Hereafter, we use the spherical uniform as the reference distribution, but 
other distributions could be chosen, depending on the applications, \cite{de2018stability,Fan2022,ghosal2021multivariate}. 
Note that if $\nu$ is an empirical measure based on random observations, it has a finite second order moment, and the estimation of the quantile map amounts to solve the OT problem \eqref{MongeOT}. 
Following \cite{Hallin-AOS_2021}, we assume, without loss of generality, that $\psi$ in Definition \ref{defQ} satisfies 
$
\psi(0)=0
$
and, for $u\in \RR^d$ such that $\Vert u \Vert = 1$,
\begin{equation}\label{hallin_psi01}
\psi(u) = \liminf\limits_{\substack{v \rightarrow u \\ \Vert v \Vert < 1}} \psi(v).
\end{equation}
Moreover, we impose, for all $u\in \RR^d$ such that $\Vert u \Vert >1$,
\begin{equation}\label{hallin_psi0}
\psi(u) = +\infty.
\end{equation}
This being said, the convex potential $\psi$ is uniquely defined over its domain $\mbox{Dom}(\psi)=\{ u \vert \psi(u) < +\infty \}$, that verifies $\mathbb{B}(0,1) \subset \mbox{Dom}(\psi) \subset \overline{\mathbb{B}}(0,1)$.
Although $\psi$ is continuous over $\mathbb{B}(0,1)$, \cite{rockafellar-1970a}[Theorem 10.3],
the gradient $\nabla\psi$ is only defined almost everywhere. 
At every $u$ where $\psi$ is not differentiable, one can still define the subdifferential
\begin{equation*}
\partial \psi(u) = \{ z \in \RR^d : \forall x\in \RR^d, \psi(x)-\psi(u) \geq \langle z, x-u \rangle \}.
\end{equation*}
Following the suitable suggestion of an anonymous referee, for all $u\in \mathbb{B}(0,1)$, we can define $Q(u)$ as the average of $\partial \psi(u)$ so that $Q$ is defined everywhere.
In fact, from \cite{Figalli2018}, as soon as $\nu$ is a continuous probability measure with non vanishing density on $\RR^d$, $\psi$ can be shown to be differentiable everywhere on $\Bd$.
The same result is showed under milder assumptions in \cite{del2019note}, that is presented hereafter.

\begin{hyp}\label{hypA}
Let $\nu$ be an absolutely continuous measure with probability density {  function} $p$ defined on its support $\XX$.
For every $R>0$, there exist {  two} constants $0<\lambda_R<\Lambda_R$ such that, for all $x \in \XX \cap \mathbb{B}(0,R)$,
\begin{equation*}
\lambda_R \le p(x) \le \Lambda_R.
\end{equation*}
\end{hyp}

\begin{hyp}\label{hypB}
The support $\XX \subset \RR^d$ of $\nu$ is convex. 
\end{hyp}

Under these assumptions, the next theorem is {  given} in \cite{del2019note}.
%
\begin{thm}[Regularity of the center-outward quantile function, \cite{del2019note}]   
Under Assumptions \ref{hypA} and \ref{hypB}, there exists a compact convex set $K$ with Lebesgue measure $0$ such that the center-outward quantile function $Q$ is a homeomorphism from $\mathbb{B}(0,1)  \backslash \{ 0 \}$ to $\XX \backslash K$, with inverse $Q^{-1}$ the center-outward distribution function. 
Moreover, $Q^{-1} = \nabla \psi^*$  where $\psi^*$  is the Fenchel-Legendre transform of $\psi$ such that $Q = \nabla \psi$, 
$$
\psi^*(x) = \supl_{u \in \mathbb{B}(0,1)} \{ \langle x,u\rangle - \psi(u) \}.
$$ 
\label{regul_Q}
\end{thm}
Hereafter, this allows us to properly define quantile contours, ranks and signs.
For the sake of simplicity, we define $\mathcal{P}_{*}(\RR^d)$ as the set of integrable probability measures for which $Q$ is a homeomorphism from $\Bd$ to its image, see Theorem \ref{regul_Q}. 
This includes any $\nu$ satisfying assumptions \ref{hypA} and \ref{hypB}, that is also integrable, to ensure finiteness of the center-outward superquantiles.
We emphasize that, for any $\nu\in \mathcal{P}_{*}(\RR^d)$, $Q$ is defined everywhere on $\Bd$, as opposed to almost everywhere on $\mathbb{B}(0,1)$.
The next definitions, taken from \cite{chernozhukov2015mongekantorovich,Hallin-AOS_2021}, gather the main concepts that we use in the following.

\begin{definition}[Quantile contours, ranks and signs]\label{def:MKquantile}
Let {  $\nu \in \mathcal{P}_*(\RR^d)$} with center-outward quantile function $Q$ and supported on $\XX \subset \RR^d$.
Then, for the distribution $\nu$, \\
(i) the quantile region $\mathds{C}_\alpha$ of order $\alpha \in [0,1]$ is the image by $Q$ of the ball $\mathbb{B}(0,\alpha)$. \\
(ii) the quantile contour $\mathcal{C}_\alpha$ of order $\alpha \in [0,1]$ is the boundary of $\mathds{C}_\alpha$. \\
(iii) the rank function $\mathcal{R}_\nu : \XX \rightarrow [0,1]$ is defined by $\mathcal{R}_\nu(x) = \Vert Q^{-1}(x) \Vert$. \\
(iv) the sign function $\mathcal{D}_\nu : \XX \rightarrow \mathbb{B}(0,1) $ is defined by $\mathcal{D}_\nu(x) = Q^{-1}(x) / \Vert Q^{-1}(x) \Vert$. 
\end{definition}
A few remarks follow from Definition \ref{def:MKquantile}. 
The $\nu$-probability of $\mathds{C}_\alpha$ is $\alpha$, by the change of variables formula for push-forward maps, which is a first requirement for quantile regions. 
In addition, one may note that the rank and sign functions require the invertibility of $Q$. 
From \cite{Hallin-AOS_2021}[Section 2], the continuity and invertibility of $Q$ outside the origin ensures the crucial fact that the quantile contours are closed and nested. 
It is the notion of center-outward ranks that allows to order points in $\XX$ relatively to $\nu$, consistently with Tukey's halfspace depth, as highlighted in \cite{chernozhukov2015mongekantorovich}. It induces the following weak order.
\begin{definition}
For {  $\nu \in \mathcal{P}_*(\RR^d)$ and} $x,y \in \XX$, we denote $x \geq_{\mathcal R} y$ and say that $y$ is deeper than $x$ if
$$
\Vert Q^{-1}(x) \Vert \ge \Vert Q^{-1}(y) \Vert.
$$
\end{definition}

The deeper a point is in $\XX$, the less extremal it is with respect to $\nu$. 
For a fixed $u \in \Bd$, we can then write $x \geq_{\mathcal{R}} Q(u)$ which enables to consider the observations ``beyond'' $Q(u)$ in some sense.
In a context where the focus is on the \textit{central} observations, $u \mapsto \EE [ X \vert X \leq_{\mathcal R} Q(u) ]$ has been introduced in \cite{Hallin_mordant_2022} 
as a center-outward Lorenz function for $X$, see Section \ref{integratedQuantiles}. 
The reverse conditional expectation $\EE [ X \vert X \geq_{\mathcal R} Q(u) ]$ might be thought of as a natural candidate for a superquantile concept.
Nevertheless, it cannot be understood as a ``typical'' observation within \textit{outward} areas. 
Indeed, consider the following example.

\begin{ex}\label{example1}
 Suppose that $\nu=U_d$, so that $Q$ is the identity. For every $u \in \Bd$, $\EE \left[ X \vert X \geq_{\mathcal R} Q(u) \right] = \EE \left[ X \big\vert \Vert X \Vert \geq \Vert u \Vert \right] $ is the expectation of a symmetric distribution over an annulus centered at the origin. Therefore, one has that $\EE [ X \vert X \geq_{\mathcal R} Q(u) ] = 0$ and it cannot be thought of as a typical extreme observation.
\end{ex}

This is caused by a lack of information in $\EE [ X \vert X \geq_{\mathcal R} Q(u) ]$. In fact, the rank function neglects the directional information of $X$, which lies in the sign function. In order to overcome this issue, we have to introduce a few notation. For any $u \in \Bd$, let
\begin{equation}\label{Lu}
L_u = \Big\{ t\frac{u}{\Vert u \Vert} : t \in [0,1] \Big\} 
\end{equation}
be a parametrization of the radius of $\Bd$ whose direction is $u/\Vert u \Vert$. The \textit{sign curve} $C_u$ associated to $u$ is the image of $L_u$ by the center-outward quantile function, namely
\begin{equation*}
C_u = Q(L_u).
\end{equation*}
Several sign curves are represented in red at the right-hand side of Figure \ref{fig:UdBananeIntro}. 
Averaging observations less deep that $Q(u)$ along the sign curve $C_u$ induces a ``typical'' value ``beyond'' $Q(u)$ in a meaningful multivariate way. 

\begin{definition}\label{superquantile}
Let $\nu \in \mathcal{P}_1(\RR^d)$ with center-outward quantile function $Q$.
The center-outward superquantile function of $\nu$ is the function $S$ defined, for any $u \in \Bd$, by
$$
S(u)= \frac{1}{1-\Vert u \Vert} \int_{\Vert u \Vert}^1 Q\Big(t\frac{u}{\Vert u \Vert}\Big) dt,
$$
where the above integral is to be understood component-wise.
\end{definition}

\begin{rem}[Consistency with the univariate case] On the real line, there is only one sign 
curve $C_1 = \{ Q( t ) ; t \in [0,1] \}$. 
Thus, definition \eqref{EandS_intro_dim1} can be seen as averaging observations less deep than $Q(\alpha)$, w.r.t. $\nu$, along the sign curve $C_1$. Nonetheless, note that center-outward quantiles slightly differ from classical quantiles in dimension $d=1$, because $U([-1,1]) \ne U([0,1])$, even if they carry the same information, \cite{Hallin-AOS_2021}[Appendix B].
\end{rem}

\noindent By the same token, we can define the center-outward expected shortfall function. 

\begin{definition}\label{expectedshortfall}
Let $\nu \in \mathcal{P}_1(\RR^d)$ with center-outward quantile function $Q$.
The center-outward expected shortfall function of $\nu$ is the function $E$ defined, for any $u \in \Bd$, by
$$
E(u)= \frac{1}{\Vert u \Vert} \int_0^{\Vert u \Vert} Q\Big(t\frac{u}{\Vert u \Vert}\Big) dt,
$$
where the above integral is to be understood component-wise.
\end{definition}
 
\begin{rem}\label{remS_E_continu}
For $\nu \in \mathcal{P}_*(\RR^d)$, $Q$ is neither defined at the origin nor at the boundary of the unit ball, unless the support of $\nu$ is compact. Hence the integrals in $S(u)$ and $E(u)$ are improper and they shall be understood respectively as
\begin{equation*}
\lim\limits_{r\rightarrow 1^-} \int_{\Vert u \Vert}^r Q\Big(t\frac{u}{\Vert u \Vert}\Big) dt
\hspace{1cm} \mbox{and} \hspace{1cm} 
\lim\limits_{r\rightarrow 0^+} \int_r^{\Vert u \Vert} Q\Big(t\frac{u}{\Vert u \Vert}\Big) dt.
\end{equation*}
However, note that they are convergent as soon as $\nu \in \mathcal{P}_1(\mathbb{R}^d)$. In fact, the necessary assumption is rather the following integrability along sign curves, 
\begin{equation}\label{RadialInteg_finite}
\int_0^1 \left\Vert Q(t\frac{u}{\Vert u \Vert}) \right\Vert dt < \infty.
\end{equation}
By the change of variables formula for push-forwards maps and by definition of $U_d$,
$\EE_\nu[\Vert X \Vert] = \EE_{U_d}[\Vert Q(U) \Vert] 
= \EE_{(R,\Phi)}[\Vert Q(R\Phi) \Vert].$
In other words, denoting by $\P_\mathbb{S}$ the uniform probability measure on the sphere $\mathbb{S}^{d-1} =  \{\varphi \in \RR^d : \Vert \varphi \Vert = 1 \}$,
\begin{equation*}
\EE_\nu[\Vert X \Vert] = \int_{\mathbb{S}^{d-1}} \int_0^1 \Vert Q(t\varphi) \Vert dt \, d\P_\mathbb{S}(\varphi).
\end{equation*}
Thus, one can see by contradiction that, as soon as $\EE_\nu[\Vert X \Vert]<\infty$, \eqref{RadialInteg_finite} holds for almost all $\varphi \in \mathbb{S}^{d-1}$, thus $S$ and $E$ must be finite almost everywhere.
\end{rem}

The following naturally extends Definition \ref{def:MKquantile}.

\begin{definition}[Superquantile and expected shortfall regions and contours]
Let $\nu \in \mathcal{P}_1(\RR^d)$  with center-outward superquantile function $S$ and expected shortfall $E$. Then, \\
(i) the superquantile (resp. expected shortfall) region $\mathds{C}^s_\alpha$ (resp. $\mathds{C}^e_\alpha$) of order $\alpha \in [0,1]$ is the image by $S$ (resp. $E$) of the ball $\mathbb{B}(0,\alpha)$. \\
(ii) the superquantile (resp. expected shortfall) contour $\mathcal{C}^s_\alpha$ (resp. $\mathcal{C}^e_\alpha$) of order $\alpha \in [0,1]$ is the boundary of $\mathds{C}^s_\alpha$ (resp. $\mathds{C}^e_\alpha$). \\
(iii) averaged sign curves by $E$ or $S$ are respectively defined by $E(L_u)$ and $S(L_u)$ for any $u\in\Bd$.
\label{superq_contours}
\end{definition}

These concepts describe point clouds through periphery or central areas, and are
illustrated in the numerical experiments carried out in Section \ref{sec:numexp}.

\subsection{Invariance properties}\label{subsec:EqInvProps}

The two following lemmas are immediate consequences of \cite[Lemmas A.7,A.8]{ghosal2021multivariate}.

%

\begin{lem}\label{invar_prop_SQ}
Assume that $X \in \RR^d$ is an integrable random vector. Suppose that $a >0$, $b\in \RR^d$ and $Y = aX+b$.
Denote by $S_X,S_Y$ and $E_X,E_Y$ their center-outward superquantile and expected shortfall functions. Then, for $u\in \Bd$, 
$$
S_Y(u) = aS_X(u) + b \quad \text{and} \quad E_Y(u) = aE_X(u) + b.
$$
\end{lem}

\noindent \textit{Proof.}
We only detail $S_Y$, as one can deduce $E_Y$ identically.  From \cite{ghosal2021multivariate}[Lemma A.7], 
\begin{align*}
S_Y(u) &= \frac{1}{1-\Vert u \Vert} \int_{\Vert u \Vert}^1 Q_Y(t \frac{u}{\Vert u \Vert}) dt,\\
&= \frac{1}{1-\Vert u \Vert} \int_{\Vert u \Vert}^1 \Big( aQ_X(t \frac{u}{\Vert u \Vert}) + b\Big) dt = {  aS_X(u) + b.} 
\tag*{$\square$}
\end{align*}

\begin{lem}\label{invar_prop_ES}
Assume that $X \in \RR^d$ is an integrable random vector. Let $S_X,S_Y$ and $E_X,E_Y$ be the center-outward superquantile and expected shortfall functions associated respectively with $X$ and $Y=AX$, for $A$ an orthonormal matrix.
Then, for $u\in \Bd$, 
\begin{align*}
S_Y(u) = AS_X(A^Tu) \quad \text{and} \quad E_Y(u) = AE_X(A^Tu).
\end{align*}
\end{lem}

\noindent \textit{Proof.}
Again, we only detail $S_Y$, as the proof for $E_Y$ is identical.
Combining \cite{ghosal2021multivariate}[Lemma A.8] with the fact that $\Vert A^T u \Vert = \Vert u \Vert$ for an orthogonal matrix, 
\begin{align*}
S_Y(u) &= \frac{1}{1-\Vert u \Vert} \int_{\Vert u \Vert}^1 Q_Y(t \frac{u}{\Vert u \Vert}) dt,\\
&= \frac{1}{1-\Vert u \Vert} \int_{\Vert u \Vert}^1 AQ_X(tA^T \frac{u}{\Vert u \Vert}) dt = {  AS_X(A^Tu).} 
\tag*{$\square$}
\end{align*}

\subsection{Main results}\label{sec:main}
 
Quoting Rockafellar \& Royset in \cite{Rockafellar2013},  
\begin{quote}
    ``the superquantile function [...] is as fundamental to a random variable as the distribution and quantile functions''.
\end{quote}
This assertion is partially motivated by the fact that the distribution, quantile and superquantile functions are uniquely determined one to another, and by the fact that pointwise convergence of these functions metrizes convergence in distribution. Such properties hold for our integrated concepts and are stated hereafter. { 
First of all}, we shall make repeated use of the fact that the center-outward superquantile and expected shortfall functions are two sides of the same coin, that is, for $u\in\Bd$,
\begin{equation}\label{E_eq_S}
\int_0^1 Q\Big(t\frac{u}{\Vert u \Vert}\Big) dt = \Vert u \Vert E(u) + (1-\Vert u \Vert) S(u).
\end{equation}
This is a generalization of the immediate property that, for the univariate setting given in the introduction with \eqref{superQdim1} and \eqref{ESdim1}, we have for all $\alpha \in ]0,1[$,
$$
\EE[X]= \alpha E(\alpha)+ (1-\alpha) S(\alpha).
$$

Our first main result is as follows.

\begin{thm}\label{caracnu1_nu2}
Let $\nu_1,\nu_2 \in \mathcal{P}_1(\RR^d)$ with respective center-outward quantile, superquantile and expected shortfall functions denoted by $Q_1,S_1,E_1$ and $Q_2,S_2,E_2$. Then, the following are equivalent.
\begin{enumerate}
\item[(i)] $\nu_1 = \nu_2$ 
\item[(ii)] $Q_1 = Q_2$ $U_d$-a.e.
\item[(iii)] $S_1 = S_2$ $U_d$-a.e.
\item[(iv)] $E_1 = E_2$ $U_d$-a.e.
\end{enumerate}
\end{thm}

\begin{proof}
First of all, it is already known that $(i) \Leftrightarrow (ii)$. Indeed, with the choice of  $U_d$ as the reference distribution, McCann's theorem \cite{McCannThm} ensures that $(i) \Rightarrow (ii)$. 
Obviously, $(ii) \Rightarrow (i)$, from the very definition of push-forward measures that is $\nu_1(A) = U_d(Q_1^{-1}(A))$ and $\nu_2(A) = U_d(Q_2^{-1}(A))$.
Hereafter, we proceed through the implication loop $ (ii) \Rightarrow (iii)  \Rightarrow (iv)  \Rightarrow (ii)$. 
Let $\P_\mathbb{S}$ be the uniform probability measure on $\mathbb{S}^{d-1} =  \{\varphi \in \RR^d : \Vert \varphi \Vert_2 = 1 \}$. 
Let $R$ and $\varphi$ be drawn uniformly on $[0,1]$ and $\mathbb{S}^{d-1}$, respectively. 
Suppose that they are independent, so that their joint distribution is given by $d\P_{(R,\varphi)}(r,\varphi) = d\P_R(r)d\P_\mathbb{S}(\varphi)$. 
There, by definition, if $f : (r,\varphi) \mapsto r\varphi$ on $[0,1]\times \mathbb{S}^{d-1}$, then $U_d = f_{\#} \P_{(R,\varphi)}$. 
Hence, from the change of variables formula for push-forwards, for any Borel set $A = f( B \times C)$, for $B \subset [0,1]$, $C \subset \mathbb{S}^{d-1}$,

\begin{align}\label{S1_S2_A}
\int_A (S_1 - S_2)(u) dU_d(u) &= \int_B \int_C (S_1 - S_2)(r\varphi) d\P_\mathbb{S}(\varphi)dr, \nonumber \\
&= \int_B \int_C \Big( \frac{1}{1-r}  \int_r^1(Q_1-Q_2)(t\varphi) dt \Big) d\P_\mathbb{S}(\varphi)dr. 
\end{align}
Nonetheless, for any $r\in B$, by the same change of variables,
\begin{align}\label{S1_S2_B}
\int_C \int_r^1(Q_1-Q_2)(t\varphi) dt d\P_\mathbb{S}(\varphi) &= \int_{f(C \times [r,1])} (Q_1-Q_2)(u)dU_d(u).
\end{align}
Note that $U_d(f(C \times [r,1])) = (1-r) \P_\mathbb{S}(C)$, which is positive since $r< 1$.
Obviously, \eqref{S1_S2_B} vanishes as soon as $Q_1 = Q_2$ $U_d$-$a.e.$
It implies that \eqref{S1_S2_A} also vanishes, which justifies that $(ii)\Rightarrow(iii)$.
Furthermore, we {  claim that $(iii)\Rightarrow(iv)$}.
Indeed, we have
\begin{equation*}\label{int01QvanishesIfS12}
\int_0^1 Q_1(t\frac{u}{\Vert u \Vert}) dt = \lim\limits_{r \rightarrow 0^+} S_1(r\frac{u}{\Vert u \Vert})
\hspace{0.65cm}\text{and}\hspace{0.65cm} \int_0^1 Q_2(t\frac{u}{\Vert u \Vert}) dt = \lim\limits_{r \rightarrow 0^+} S_2(r\frac{u}{\Vert u \Vert}).
\end{equation*}
Consequently, if we assume that $S_1 = S_2$ $U_d$-$a.e.$, we obtain that
\begin{equation}\label{int01QvanishesIfSdoes}
\int_0^1 (Q_1-Q_2)(t\frac{u}{\Vert u \Vert}) dt = 0.
\end{equation}
Using \eqref{E_eq_S}, the desired result follows, $(iii)\Rightarrow(iv)$.
Finally, assume that $E_1 = E_2$ $U_d$-a.e. 
Consequently, for all $r \in ]0,1[$ and for all $\varphi \in \mathbb{S}^{d-1}$, 
$$
\int_0^r Q_1(t\varphi) dt = \int_0^r Q_2(t\varphi) dt.
$$
Using that $ \int_a^b = \int_0^b - \int_0^a$, for any $0\le a \le b \le 1$,
\begin{equation}\label{intabQdt}
\int_a^b Q_1(t\varphi) dt = \int_a^b Q_2(t\varphi) dt.
\end{equation}
For any measurable $B \subset \mathbb{S}^{d-1}$, by integrating \eqref{intabQdt} w.r.t. $\P_{\mathbb{S}}$, 
\begin{equation*}\label{intQdPrPvarphi}
\int_B \int_a^b Q_1(t\varphi) d\P_R(t)d\P_\mathbb{S}(\varphi) = \int_B \int_a^b Q_2(t\varphi) d\P_R(t)d\P_\mathbb{S}(\varphi).
\end{equation*}
By use of $U_d = f_{\#} \P_{(R,\varphi)}$, a change of variables above yields, for any $A \subset \Bd$,
$$
\int_A Q_1(u) dU_d(u)  = \int_A Q_2(u) dU_d(u),
$$
and the result follows.

\end{proof}

Interestingly enough, our proposed integrated quantile functions are both simply related to the Kantorovich potential. Somehow, this development generalizes the work of \cite{rockafellar2014random} 
where the distribution function is related to the univariate superquantile by the way of the surexpectation function, which is nothing more than a particular primitive of the distribution function. 

\begin{prop}\label{psi_is_E}
The center-outward expected shortfall function of $\nu \in \mathcal{P}_1(\RR^d)$ satisfies, for any $u\in \Bd$,
\begin{equation}\label{eq:psi_is_E}  
\langle E(u),u \rangle = \psi(u).
\end{equation}
Moreover, the center-outward superquantile function of a compactly supported probability measure $\nu\in \mathcal{P}_*(\RR^d) $ verifies, for every $u\in \Bd$, 
\begin{equation}\label{psi_and_S}
\langle S(u), u \rangle = \frac{\Vert u \Vert}{1 - \Vert u \Vert}\Big(\psi( \frac{u}{\Vert u \Vert}) - \psi(u)\Big).
\end{equation}
\end{prop}

\begin{proof} 

{ 
Fix $u\in\Bd$ and let $f(t) = \psi(tu)$. Then $f$ is a finite and convex function on 
$]-a,a[$ for some $a>1$, so that one can apply \cite{rockafellar-1970a}[Corollary 24.2.1]. 
Combined with $a.e.$ differentiability, \cite{rockafellar-1970a}[Theorem 25.3], this yields
\begin{equation}
\psi(u) = \int_0^1 \langle \nabla \psi(tu),u\rangle dt,
\end{equation}}
%
which can be rewritten as 
$$
\psi(u) = \langle \int_0^1 Q(tu) dt,u \rangle = \langle E(u),u \rangle,
$$
from a simple change of variables, leading to our first point \eqref{eq:psi_is_E}.
Moreover, denote for {  any} $u \in \Bd$, 
$$
f(t) = \psi(t\frac{u}{ \Vert u \Vert}).
$$
This function is finite on $[0,1]$ but not on a larger interval anymore. 
But, because $\nu \in \mathcal{P}_*(\RR^d)$, $f$ is differentiable everywhere on $]0,1[$, and we have by the chain rule formula, 
$$
f'(t)=\langle Q(t\frac{u}{ \Vert u \Vert}),\frac{u}{ \Vert u \Vert}\rangle.
$$
Since $f'(t)$ is a non-decreasing function from $\RR$ to $[-\infty,+\infty]$ finite at $t=\Vert u \Vert$, \cite{rockafellar-1970a}[Theorem 24.2] ensures that the function $F$, defined for all $x\in\RR$, by
$$
F(x) = \int_{\Vert u \Vert}^x f'(t) dt,
$$ 
is a well-defined closed proper convex function on $\mathbb{R}$.
We emphasize that $F(x)$ takes infinite values as soon as $x>1$, while, for $x = 1$, $F$ is well-defined as a Lebesgue integral, or as a limit of Riemann integrals, as explained in the proof of \cite{rockafellar-1970a}[Theorem 24.2], and it may take finite or infinite values.
In addition, the latter theorem tells us that $F(x)=f(x)+\alpha$ for some $\alpha \in \RR$ everywhere. 
But $F(\Vert u \Vert) = 0$ and $f(\Vert u \Vert) = \psi(u)$, so $F(x) = f(x) - \psi(u)$. 
Assuming that the support $\XX$ of $\nu$ is compact, $\psi$ must be lipschitz continuous, \cite{mccann2001polar}[Lemma 2], a fortiori bounded and the subdifferential $\partial \psi (K)$ must also be bounded, see $e.g.$ \cite{rockafellar1981favorable}. 
As a byproduct, $\psi(u / \Vert u \Vert)$ and $Q(u / \Vert u \Vert)$ are finite.
In this case, $F(1) = f(1)-\psi(u)$ is finite and can be rewritten as
$$
F(1) = \int_{\Vert u \Vert}^1 \langle Q(t\frac{u}{ \Vert u \Vert}),\frac{u}{ \Vert u \Vert}\rangle dt = \psi(\frac{u}{ \Vert u \Vert})  -  \psi( u),
$$
which implies \eqref{psi_and_S},  completing the proof of Proposition \ref{psi_is_E}.

\end{proof}

Thanks to this relation between the potential $\psi$ whose gradient gives $Q$ and the center-outward expected shortfall function, we are now able to characterize the convergence in distribution for a sequence of random vectors through superquantiles and expected shortfalls. 
For that purpose, we rely on existing results on the relation between convergence in distribution and center-outward quantile functions. 

\begin{lem}
\label{conv_pot}
Let $(X, X_n)$ be a sequence of random vectors with distributions $\nu$ and $\nu_n$ respectively, in $\mathcal{P}_*(\RR^d)$. 
Then, 
\begin{eqnarray}
\hspace{1cm} X_n \overset{\mathcal{L}}{\rightarrow} X 
&\Leftrightarrow &
{  \forall u \in \mathbb{B}(0,1), \, \lim_{n \rightarrow \infty} \psi_n(u) = \psi(u),} \label{pontwise_conv_pot0}\\
&\Leftrightarrow &
\lim_{n \rightarrow \infty}\supl_{u \in K} \vert \psi_n(u) - \psi(u)\vert=0, \label{pontwise_conv_pot} \\
&\Leftrightarrow & 
\lim_{n \rightarrow \infty}\supl_{u \in K} \Vert Q_n(u) - Q(u) \Vert =0, \label{rockaf_thm1}
\end{eqnarray} 
for every compact $K\subset \Bd$. {  In fact, uniform convergence of $\psi_n$ towards $\psi$ even holds on every compact $K\subset \mathbb{B}(0,1)$.}

\end{lem}

\begin{proof} 
On the one hand, assume that $(X_n)$ converges in distribution to $X$. 
Then, the {  right-hand} side of \eqref{rockaf_thm1} is the main result of \cite{ghosal2021multivariate}[Theorem 4.1]  when $Q_n$ and $Q$ are homeomorphisms between convex sets, with uniform convergence on any compact subset. 
But, with reference distribution $U_d$, center-outward quantile maps are not, because of the discontinuity at the origin, see \cite{Hallin-AOS_2021}[Remark 3.4]. 
Nonetheless, as highlighted in the proof of \cite{Hallin-AOS_2021}[Proposition 3.3], the assumption that $X_n \overset{\mathcal{L}}{\rightarrow} X$ for $\nu_n, \nu\in\mathcal{P}_*(\RR^d)$ is sufficient to apply Theorem 2.8 in \cite{delBarrio2017}. 
As a consequence, $\psi_n(u)$ converges towards $\psi(u)$ for every $u\in\mathbb{B}(0,1)$, that is the right-hand side of \eqref{pontwise_conv_pot0}. 
Being finite and convex functions, \cite{rockafellar-1970a}[Theorem 10.8] ensures that such pointwise convergence implies the right-hand side of \eqref{pontwise_conv_pot}. 
Since $\psi_n$ and $\psi$ are differentiable on $\Bd$ as soon as $\nu_n,\nu \in \mathcal{P}_*(\RR^d)$, one can apply \cite{rockafellar-1970a}[Theorem 25.7] on any open and convex set $K'\subset \Bd$. 
This gives us the uniform convergence of $Q_n$ towards $Q$ for any compact $K$ included in an open and convex set $K'\subset \Bd$. 
To extend this to any compact $K$, we use that, by compacity, one can extract from the open cover $\{\mathbb{B}(x,\delta) ; x\in K\}$ a finite cover. 
Consequently,  it exists $N\in\mathbb{N}$ and $x_1,\cdots,x_N\in K$ such that
\begin{equation}\label{almost_supcvgceQn}
\sup\limits_{u\in K} \Vert Q_n(u) - Q(u) \Vert
\leq 
\sum_{k=1}^N \sup\limits_{u\in \mathbb{B}(x_{k},\delta)} \Vert Q_n(u) - Q(u) \Vert.
\end{equation}
Here, the closure of each ball $\mathbb{B}(x_{k},\delta)$ is a compact set, which, because of the choice of $\delta$, is a subset of some open and convex set $K'\subset \Bd$.  
But we already know that $Q_n$ uniformly converges towards $Q$ on such a set, thus the right-hand side of \eqref{almost_supcvgceQn} vanishes when $n\rightarrow 0$, which yields the right-hand side of \eqref{rockaf_thm1}.
It only remains to prove that {  $X_n \overset{\mathcal{L}}{\rightarrow} X$  as a direct consequence.}
This last claim relies on the Portmanteau theorem which says that
$(X_n)$ converges in distribution to $X$ iff for any bounded and continuous function $f$, 
\begin{equation}\label{carac_i}
\lim\limits_{n \rightarrow + \infty}\EE[ f(X_n) ] = \EE[f(X)].
\end{equation}
However, we clearly have for any bounded and continuous function $f$, 
\begin{align*}
\EE[ f(X_n) ]  &= \int_{\RR^d} f(x) d\nu_n(x) = \int_{\mathbb{B}(0,1)} f(Q_n(u)) dU_d(u).
\end{align*}
using the change of variables $\nu_n = Q_{n\#} U_d$. Hence, as
$f \circ Q_n$ is uniformly bounded, the dominated convergence theorem leads to \eqref{carac_i}.
\end{proof}

\begin{thm}\label{expshort_metrizes_weakcvgce}
Let $(X, X_n)$ be a sequence of random vectors with distributions $\nu$ and $\nu_n$ respectively, in $\mathcal{P}_*(\RR^d)$. 
Then, 
\begin{eqnarray}
\label{cvgcePointwiseE_n}
\hspace{1cm} X_n \overset{\mathcal{L}}{\rightarrow} X 
&\Leftrightarrow &
\forall u \in \mathbb{B}(0,1)\backslash \{0\}, \lim\limits_{n\rightarrow +\infty} E_n(u) = E(u), \\
&\Leftrightarrow & 
\lim_{n \rightarrow \infty}\supl_{u \in K} \Vert E_n(u) - E(u) \Vert =0, \label{pointwisecvcgeES}
\end{eqnarray} 
$\text{for every compact } K \subset \Bd$.

\end{thm}

\begin{proof}
On the one hand, assume that $(X_n)$ converges in distribution to $X$. Then, it follows from
\eqref{rockaf_thm1} that $(Q_{n})$ converges uniformly to $Q$ on any compact
$K \subset \Bd$. 
{  For any $u \in \Bd$, we have from Definition \ref{expectedshortfall} that
\begin{align*}
\left\Vert E_{n}(u) - E(u) \right\Vert &\leq \frac{1}{\Vert u \Vert} \int_0^{\Vert u \Vert} R_n(t \frac{u}{\Vert u \Vert }) dt,
\end{align*}
where $R_n(v) = \left\Vert Q_{n}(v) - Q(v) \right\Vert.$
Fix a compact $K\subset \Bd$ such that for every $u\in K$, $C<\Vert u \Vert<D$ for some positive constants $C,D$. Then, we clearly have
\begin{align*}
\sup_{u\in K} \left\Vert E_{n}(u) - E(u) \right\Vert &\leq \frac{1}{C} \int_0^{D} \sup_{u\in K}  R_n(t \frac{u}{\Vert u \Vert }) dt.
\end{align*}
Consequently, for any $\xi\in \RR$ such that $0\leq \xi<D$, 
\begin{align}
\hspace{-0.2cm} \sup_{u\in K} \left\Vert E_{n}(u) - E(u) \right\Vert 
&\leq \frac{1}{C}\left( \int_0^{\xi} \hspace{-0.1cm} \sup_{u\in K} R_n(t \frac{u}{\Vert u \Vert }) dt + \int_{\xi}^{D} \hspace{-0.1cm}\sup_{u\in K} R_n(t \frac{u}{\Vert u \Vert }) dt\right) \hspace{-0.1cm}.
\label{schalesDecompSupEnE}
\end{align}
On the one hand, for all $t \in [\xi,D]$ and $u\in K$, $t u / \Vert u \Vert$ lies inside a compact $K' \subset \Bd$, so that the second term in \eqref{schalesDecompSupEnE} satisfies
\begin{equation}\label{cvgceEnpart1}
\int_{\xi}^{D} \sup_{u\in K} R_n(t \frac{u}{\Vert u \Vert }) dt 
\leq (D-\xi) \sup_{v\in K'} R_n(v) 
\leq \sup_{v\in K'} R_n(v),
\end{equation}
which vanishes when $n\rightarrow +\infty$. On the other hand,
\begin{equation}\label{cvgceEnpart2}
\int_0^{\xi} \sup_{u\in K} R_n(t \frac{u}{\Vert u \Vert }) dt \leq \xi \sup_{v\in \overline{\mathbb{B}}(0,D)} R_n(v).
\end{equation}
We now claim that $\sup_{v\in \overline{\mathbb{B}}(0,D)} R_n(v)$ is bounded by a constant.
Recall that, for any $v\in \overline{\mathbb{B}}(0,D)$, $Q_n(v)$ and $Q(v)$ belong to the subdifferentials $\partial \psi_n(v)$ and $\partial \psi(v)$, by definition. 
 Because $S = \overline{\mathbb{B}}(0,D)$ is a compact subset of the open unit ball, and because, for any $u \in S$, $\psi_n(u)$ is convergent and a fortiori bounded, see Lemma \ref{conv_pot}, we are in position to apply \cite{rockafellar-1970a}[Theorem 10.6]. 
 It directly implies that the sequence $(\psi_n)$ is uniformly bounded and equi-Lipschitz on $S$. 
 However, it is well-known that any $L$-lipschitz convex function $\phi$ on $S$ must have a bounded subdifferential $\partial \phi(S)$, see $e.g.$ \cite{rockafellar1981favorable}. 
Indeed, for any $u\in S$ and any $\xi \in \partial \phi(S)$, consider $\delta >0$ such that $y = u - \delta \xi$ belongs to $S$. By definition of $\partial \phi(S)$ and the lipschitz constant $L$, 
 $$
 L \delta \Vert \xi \Vert \geq \phi(y) - \phi(u) \geq \langle \xi, y - u \rangle =  \delta \Vert \xi \Vert^2.
 $$
 This immediately gives us the existence of a uniform bound on the family of subdifferentials $(\partial \psi_n(S))_n$. 
 Because $\partial \psi(S)$ is also bounded, by lipschitz regularity of $\psi$ on $S$, 
\begin{equation}\label{rockafThm10.6}
 \sup_{v\in \overline{\mathbb{B}}(0,D)} R_n(v) \leq  \sup_{v\in \overline{\mathbb{B}}(0,D)} \Vert Q_n(v)\Vert + \Vert Q(v) \Vert \leq M < +\infty.
\end{equation}
Therefore the right-hand side of \eqref{cvgceEnpart2} can be bounded by $\xi M$ for $M$ given in \eqref{rockafThm10.6}.
Then it follows from \eqref{schalesDecompSupEnE} and \eqref{cvgceEnpart1} that, for any $\xi \in [0,D[$, it exists $n_1 \in \mathbb{N}$ such that, for all $n\geq n_1$,
\begin{equation*}
\sup_{u\in K} \left\Vert E_{n}(u) - E(u) \right\Vert \leq \frac{1}{C}\Big( \xi M +\sup_{v\in K'} R_n(v)\Big).
\end{equation*}
Thus, for any $\xi \in [0,D[$, 
\begin{equation*}
\lim\limits_{n\rightarrow +\infty}\sup_{u\in K} \left\Vert E_{n}(u) - E(u) \right\Vert \leq \frac{\xi M}{C}.
\end{equation*}
which leads to the right-hand side of \eqref{pointwisecvcgeES} as $\xi$ goes to zero.}
In addition, the right-hand side of \eqref{pointwisecvcgeES} immediately implies
the right-hand side of \eqref{cvgcePointwiseE_n}. On the other hand, assume that
the right-hand side of \eqref{cvgcePointwiseE_n} holds. 
Then, we obtain from \eqref{eq:psi_is_E} that $\forall u \in \mathbb{B}(0,1)$,
$$ 
 \lim\limits_{n\rightarrow +\infty} \psi_n(u) = \psi(u).
$$ 
Thus, the desired result follows from Lemma \ref{conv_pot}, which completes the proof of Theorem \ref{expshort_metrizes_weakcvgce}. 
\end{proof}

Our last result requires some uniform integrability assumption on the sequence $(X_n)$.

\begin{thm}\label{cvgceSk}
Let $(X, X_n)$ be a sequence of random vectors with distributions $\nu$ and $\nu_n$ respectively, in $\mathcal{P}_*(\RR^d)$.  
In addition, suppose that there exists a random variable $Z$ greater than $1$ such that $\EE [Z \ln(Z)] < +\infty$ and for all $n\in \mathbb{N}$,
\begin{equation}\label{VertXnZ}
\Vert X_n\Vert \leq Z \hspace{1cm} a.s.
\end{equation}
Then,
\begin{equation}\label{convSnS}
X_n \overset{\mathcal{L}}{\rightarrow} X 
\Leftrightarrow 
\lim_{n \rightarrow \infty} S_n(u) = S(u),
\end{equation}
for almost all $u\in \mathbb{B}(0,1)$.
\end{thm}

\begin{rem}
One can observe that if the support of every $\nu_n$ is bounded by some constant $M$, \eqref{VertXnZ} is no longer needed. Under this restrictive assumption, there is no mass going out to infinity in any direction and one can easily check that 
$$
X_n \overset{\mathcal{L}}{\rightarrow} X 
\Leftrightarrow 
\lim_{n \rightarrow \infty} \supl_{u\in K} \Vert S_n(u) - S(u) \Vert = 0,
$$
for every compact $K \subset \Bd$.
\end{rem}

\begin{proof}
On the one hand, assume that $(X_n)$ converges in distribution to $X$. We clearly have from \eqref{VertXnZ} that, if $\Phi(x) = x \ln(x)$, then
\begin{equation*}
\EE [ \; \supl_{n \in \mathbb{N} } \Phi(X_n) ] \leq \EE [Z \ln(Z)] < +\infty.
\end{equation*}
Hence, using the same change of variables as in the proof of Theorem \ref{caracnu1_nu2}, we obtain that 
\begin{equation}\label{intUdQnfinite}
\EE [ \; \supl_{n \in \mathbb{N} } \Phi(X_n) ] = \int_{\mathbb{S}^{d-1}} \int_0^1 \supl_{n \in \mathbb{N} } \Vert Q_n(t\varphi) \Vert \ln(\Vert Q_n(t\varphi)\Vert) dt d\P_{\mathbb{S}}(\varphi) <+\infty,
\end{equation}
where we recall that $\mathbb{S}^{d-1} = \{ \varphi \in \RR^d : \Vert \varphi \Vert_2 = 1\}$. Consequently, we deduce from \eqref{intUdQnfinite} that for almost all $\varphi \in \mathbb{S}^{d-1}$, 
\begin{equation}\label{int_01_supQ_finite}
\supl_{n \in \mathbb{N}} \int_0^1 \Vert Q_n(t\varphi) \Vert \ln(\Vert Q_n(t\varphi)\Vert) dt \leq \int_0^1 \supl_{n \in \mathbb{N} } \Vert Q_n(t\varphi) \Vert \ln(\Vert Q_n(t\varphi)\Vert) dt <+\infty.
\end{equation}
It clearly leads to the uniform integrability of $(Q_n)$ along sign curves, see \cite{Bogachev2007}[Theorem 4.5.9]. However, we already saw from Lemma \ref{conv_pot} that the convergence in distribution of $(X_n)$ to $X$ implies the pointwise convergence of $Q_n$ to $Q$ on $\Bd$. Therefore, it follows from the Lebesgue-Vitali theorem \cite{Bogachev2007}[Theorem 4.5.4] that for {  almost} all $\varphi \in \mathbb{S}^{d-1}$, 
$$
\lim\limits_{n \rightarrow +\infty} \int_0^1 Q_n(t\varphi) dt = \int_0^1 Q(t\varphi) dt, 
$$
which means that for {  almost} all $u \in \Bd$, 
\begin{equation}\label{limIntQn}
\lim\limits_{n \rightarrow +\infty} \int_0^1 Q_n \Big(t\frac{u}{\Vert u \Vert}\Big) dt = \int_0^1 Q\Big(t\frac{u}{\Vert u \Vert}\Big) dt.
\end{equation}
Hereafter, we already saw from \eqref{E_eq_S} that 
\begin{equation}\label{eqSn_S}
S_n(u) - S(u) = \frac{1}{1-\Vert u \Vert} \left( \int_0^1 (Q_n-Q)\Big(t \frac{u}{\Vert u \Vert} \Big) dt - \Vert u \Vert \Big( E_n(u) - E(u) \Big)\right).
\end{equation}
Finally, the right-hand side of \eqref{convSnS} follows from \eqref{cvgcePointwiseE_n} together with \eqref{limIntQn} and \eqref{eqSn_S}. On the other hand, assume that the right-hand side of \eqref{convSnS} holds. 
We have for {  almost} all $u\in \Bd$,
\begin{equation}\label{limIntQnVertU}
\lim\limits_{n \rightarrow +\infty} \int_{\Vert u \Vert}^1 Q_n \Big(t\frac{u}{\Vert u \Vert}\Big) dt = \int_{\Vert u \Vert}^1 Q\Big(t\frac{u}{\Vert u \Vert}\Big) dt.
\end{equation}
Hence, we obtain from \eqref{limIntQnVertU} {  that for almost every $r\in]0,1[$ and $\varphi \in \mathbb{S}^{d-1}$,
}
\begin{equation}\label{limIntQnR}
\lim\limits_{n \rightarrow +\infty} \int_{r}^1 Q_n (t\varphi ) dt = \int_{r}^1 Q(t \varphi) dt.
\end{equation}
Obviously, for {  almost} all $r\in]0,1[$ and {  almost} all $\varphi \in \mathbb{S}^{d-1}$,
\begin{align*}
\lim\limits_{n \rightarrow +\infty} \int_{0}^1 (Q_n-Q) (t\varphi ) dt &= \lim\limits_{n \rightarrow +\infty}\left( \int_0^r (Q_n-Q) (t\varphi ) dt +  \int_r^1 (Q_n-Q) (t\varphi ) dt\right) ,\\
&= \lim\limits_{n \rightarrow +\infty} \int_0^r (Q_n-Q) (t\varphi ) dt.
\end{align*}
In addition, this integral is always finite {  since \eqref{RadialInteg_finite} holds for integrable probability measures}. As the left-hand side does not depend on $r$,
\begin{equation}\label{limr0}
\lim\limits_{n \rightarrow +\infty} \int_{0}^1 (Q_n-Q) (t\varphi ) dt = \lim\limits_{r\rightarrow 0} \lim\limits_{n \rightarrow +\infty}  \int_0^r (Q_n-Q) (t\varphi ) dt.
\end{equation}
From the uniform integrability of $(Q_n)$ along sign curves, we obtain that 
$$
\lim\limits_{r \rightarrow 0 } \lim\limits_{n \rightarrow +\infty } \int_0^r \Vert Q_n(t\varphi) - Q(t\varphi) \Vert dt = 0.
$$
Hence, by use of \eqref{limr0}, we find that 
$$
\lim\limits_{n \rightarrow +\infty} \int_0^1 Q_n(t\varphi) dt = \int_0^1 Q(t\varphi) dt.
$$
It ensures that for {  almost} all $u\in \Bd$,
\begin{equation}\label{limint01Qu}
\lim\limits_{n \rightarrow +\infty}  \int_0^1 Q_n\Big(t\frac{u}{\Vert u \Vert} \Big) dt = \int_0^1 Q\Big(t\frac{u}{\Vert u \Vert} \Big) dt.
\end{equation}
Consequently, we deduce from \eqref{eqSn_S} and \eqref{limint01Qu} that for {  a.e.} $u\in\Bd$,
\begin{equation}\label{ThmSn_En_E}
\lim\limits_{n \rightarrow +\infty}  E_n(u) = E(u)
\end{equation}
Finally, it follows from \eqref{ThmSn_En_E} together with \eqref{cvgcePointwiseE_n} that $(X_n)$ converges in distribution to $X$, which completes the proof of Theorem \ref{cvgceSk}. 
\end{proof}

\section{A class of reference measures}\label{class_ref}

Until now, the concepts of MK quantiles, superquantiles and expected shortfalls strongly rely on the choice of the reference distribution, and one may question the choice of $U_d$. 
This issue is discussed hereafter, with explicit MK quantile functions for generalized gamma models, by calibrating the reference distribution. 
Ultimately, one obtains nested regions indexed by their probability content.
The choice of the reference distribution $\mu$ is a major tool to adapt MK quantiles, ranks and signs to any task encountered.   
Notably, the uniform distribution over the unit hypercube has advantages concerning marginal independence, \cite{deb2023multivariate,ghosal2021multivariate}, whereas orthogonal invariance only holds with a spherical reference distribution, \cite{ghosal2021multivariate,Hallin_mordant_2022}. 
$U_d$ plays a specific role for the interpretation of quantile regions indexed by a probability content level. 
This section deals with an alternative class of reference measures defined by $\alpha$-level sets on the $p$-unit ball, that preserve the interpretability of $U_d$.
We also consider the restriction of $U_d$ to $\RR^d_+$, with a \textit{left-to-right} ordering that might be preferred in the subsequent risk applications to the \textit{center-outward} one.
Denote by $\mathbb{S}^{d,p}$ the unit sphere $\mathbb{S}^{d,p}=\{\varphi \in \RR^d : \Vert \varphi \Vert_p = 1 \}$ with 
$$ 
\Vert \varphi \Vert_p^p = \sum_{k=1}^d \vert \varphi_k \vert^p,
$$
and $\mathbb{S}^{d,p}_+ = \mathbb{S}^{d,p}\cap\RR^d_+$. 
 Moreover, let $q$ be the Holder conjugate of $p$, given by
$$
\frac{1}{p}+\frac{1}{q}=1.
$$

\begin{definition}\label{p_sphericunif}
The $q$-spherical conjugate distribution $U_{d,q}$ is defined as the product $R\Phi$ between two independent random variables $R$ and $\Phi$ where 
\begin{itemize}
\item $R$ has uniform distribution on $[0,1]$, 
\item  the distribution of $\Phi\in \mathbb{S}^{d,q}$ is given by $\Phi = \Psi^{\otimes (p-1)}$ for $\Psi$ uniformly drawn on $\mathbb{S}^{d,p}$, and $\otimes$ the component-wise exponent.
\end{itemize}
The restriction of $U_{d,q}$ to the convex cone $\RR^d_+$ is denoted by $U_{d,q}^+$.
\end{definition}
With reference measure $U_{d,q}$, the MK quantile function $Q_q$ can always be defined almost everywhere from a convex potential $\psi$ by Definition \ref{defQ}, for which we can always assume that $\psi(0)=0$.
At any point $u$ of non differentiability of $\psi$, one can still define $Q_q(u)$ as the average of the subdifferential $\partial \psi(u)$, so that $Q_q$ is defined everywhere on the $q$-unit ball. 
The interpretability of the quantile concepts remains the same among this class, as the $U_{d,q}$-probability of the $q$-ball of radius $\alpha \in [0,1]$ is $\alpha$, 
\begin{equation*}
\mathbb{P}(\Vert R \Phi \Vert_q \leq \alpha) = \mathbb{P}( R \leq \alpha) = \alpha.
\end{equation*}
Crucially for us, our definitions and properties of superquantiles and expected shortfalls naturally adapt. 
Definition \ref{superquantile} as well as Definition \ref{expectedshortfall} extend, for all $u\in\mathbb{B}(0,1)$, to 
$$
S_q(u) = \frac{1}{1-\Vert u \Vert_q} \int_{\Vert u \Vert_q}^1 Q_q(t \frac{u}{\Vert u \Vert_q}) dt
$$ 
and
$$
E_q(u) = \frac{1}{\Vert u \Vert_q} \int_0^{\Vert u \Vert_q} Q_q(t \frac{u}{\Vert u \Vert_q}) dt.
$$

We now develop on how to sample uniformly on $p$-spheres $\mathbb{S}^{d,p}$.
Let $\Gamma$ stand for the Euler Gamma function and denote by $\mathcal{L}_p$ 
 the probability distribution with density function
$$
f_p(x) = \frac{p}{\Gamma(p^{-1})} \exp(-x^p) I_{\RR_+}(x).
$$
A random variable $X$ on $\RR_+$ is drawn from $\mathcal{L}_p$ if and only if $X^p$ follows a Gamma distribution with shape and scale parameters $1/p$ and $1$. The following lemma indicates how to sample uniformly on $p$-spheres when $p>0$.

\begin{lem}[\cite{Barthe2005,schechtman1990volume}]\label{lemschechtman}
For any real $p>0$, the components of a random vector $X\in\RR^d_+$ are independent with distribution $\mathcal{L}_p$ if and only if the random variables $\Vert X \Vert_p$ and $\Vert X \Vert_p^{-1}X$ are independent where 

\begin{enumerate}
\item[-] \hspace{-0.2cm} $\Vert X \Vert_p^{-1}X$ is uniformly distributed on $\mathbb{S}^{d,p}_+$,\\
\item[-] \hspace{-0.2cm} $\Vert X \Vert_p^p$ has Gamma distribution with shape and scale parameters $d/p$ and $1$.
\end{enumerate}
\end{lem}

\noindent For $p=2$, $\mathcal{L}_p$ corresponds to $\vert Z\vert$ where $Z$ is drawn from a $\mathcal{N}(0,1)$ distribution.
The MK distribution function with reference $U_d$ (or $U_{d,2}$) has been obtained in \cite{chernozhukov2015mongekantorovich}[Section 2.4]. 
Here, we use similar arguments to find gradient-of-convex maps for our alternative class of reference measures.
We emphasize that such maps are invariant to shifts, see \cite{ghosal2021multivariate}[Lemma A.7], whereas invariance to rotations requires a spherically symmetric reference measure, \cite{ghosal2021multivariate}[Lemma A.8].

\begin{prop}\label{prop_explicit_Gamma}
Suppose that the components of a random vector $X\in\RR^d_+$ are independent with $\mathcal{L}_p$ distribution for $p>1$, and let $q = p/(p-1)$.
Then, the MK distribution function with respect to $U_{d,q}^+$ has the explicit formulation 
\begin{equation}\label{distribFgamma}
F_q(x) = (\Vert x \Vert_p^{-1} x)^{\otimes (p-1)} G(\Vert x \Vert_p),
\end{equation}
where $\otimes$ stands for the component-wise exponent and $G$ is the univariate distribution function of $\Vert X \Vert_p$. Its inverse, the MK quantile function, is given, for all $u\in\RR^d_+$ such that $\Vert u \Vert_q \leq 1$, by
\begin{equation}\label{quantilQgamma}
Q_q(u) = (\Vert u \Vert_q^{-1} u)^{\otimes (q-1)} G^{-1}(\Vert u \Vert_q).
\end{equation}
The MK superquantile and expected shortfall functions are respectively given,  for all $u\in\RR^d_+$ such that $\Vert u \Vert_q \leq 1$, by
\begin{equation}\label{SQgamma}
S_q(u) = (\Vert u \Vert_q^{-1} u)^{\otimes (q-1)} \overline{S}(\Vert u \Vert_q),
\end{equation}
\begin{equation}\label{ESgamma}
E_q(u) = (\Vert u \Vert_q^{-1} u)^{\otimes (q-1)} \overline{E}(\Vert u \Vert_q),
\end{equation}
where $\overline{S}$ and $\overline{E}$ are the univariate superquantile and expected shortfall functions associated with the distribution of $\Vert X \Vert_p$.
\end{prop}

\begin{proof}
Denote by $G$ the probability distribution function of $\Vert X \Vert_p$. Moreover, for $z\in \RR$, let 
$$
\Psi(z) = \int_{-\infty}^z G(t) dt,
$$
and $\varphi(x) = \Psi(\Vert x \Vert_p)$. 
Then, $\varphi$ is convex by the composition between the non-decreasing function $\Psi$ and $ \Vert \cdot \Vert_p$, both convex, see \cite{Rockafellar2013}[Theorem 5.1].
In addition, $\nabla \varphi(x) = F_q(x)$, which means that $F_q$ is the gradient of a convex function.
It remains to show that $F_q(X)$ follows the distribution $U_{d,q}^+$.
From Lemma \ref{lemschechtman}, $F_q(X)$ is the product of two independent random variables, $G(\Vert X \Vert_p)$ and $(\Vert X \Vert_p^{-1} X)^{\otimes (p-1)}$.
On the one hand, $G(\Vert X \Vert_p)$ is uniformly distributed on $[0,1]$, by definition of $G$. 
On the other hand, the distribution of $\Vert X \Vert_p^{-1} X$ is uniform on $\mathbb{S}^{d,p}_+$, which implies \eqref{distribFgamma}.
In order to compute $Q_q=F_q^{-1}$, it remains to invert \eqref{distribFgamma}.
If $u = (\Vert x \Vert_p^{-1} x)^{\otimes (p-1)} G(\Vert x \Vert_p)$, 
$$
\Vert u \Vert_q^q = G(\Vert x \Vert_p)^q \sum_{k=1}^d \Big(\frac{x_k}{\Vert x \Vert_p}\Big)^{p}.
$$
Thus, $\Vert u \Vert_q = G(\Vert x \Vert_p)$ which yields $\Vert u \Vert_q^{-1} u = (\Vert x \Vert_p^{-1} x)^{\otimes (p-1)}$. 
This rewrites $\Vert u \Vert_q^{-1} u \Vert x \Vert_p^{p-1} = x^{\otimes (p-1)} $, where $\Vert x \Vert_p^{p-1} = G^{-(p-1)}(\Vert u \Vert_q)$, so
$$
\Vert u \Vert_q^{-1} u G^{-(p-1)}(\Vert u \Vert_q) = x^{\otimes (p-1)}.
$$
Finally, \eqref{quantilQgamma} follows by applying the exponent $q-1= 1/(p-1)$, which implies \eqref{SQgamma} and \eqref{ESgamma}, completing the proof of Proposition \ref{prop_explicit_Gamma}.

\qedhere
\end{proof}

\begin{rem}
It is well-known that in the special case $p=2$, the distribution of $(\Vert X \Vert_p^{-1} X)^{\otimes (p-1)}$ is uniform on $\mathbb{S}^{d,q}_+$ where $q=p=2$, see \cite{chernozhukov2015mongekantorovich}.
One may wonder if this property holds true for other choices of $p>1$.
This is actually not the case as we shall see now.
As $q$ is the Holder conjugate of $p$, we clearly have
$\Vert X \Vert_p^{p-1} = \Vert X^{\otimes (p-1)} \Vert_q$ which implies that
\begin{equation}\label{lemschechtman}
\left( \frac{ X }{ \Vert X \Vert_p} \right)^{\otimes (p-1)} = \frac{ X^{\otimes (p-1)} }{\Vert X^{\otimes (p-1)} \Vert_q}.
\end{equation}
All the components of $X$ share the same $\mathcal{L}_p$ distribution. Consequently, each  
$X_i^p$ has a Gamma distribution with shape and scale parameters $1/p$ and $1$.
It implies that $X_i^{p-1}=X_i^{p/q}$ is the power $1/q$ of a Gamma distribution.
We immediately deduce from Lemma \ref{lemschechtman} that, as soon as $p \neq 2$, $(\Vert X \Vert_p^{-1} X)^{\otimes (p-1)}$ is not uniformly distributed on $\mathbb{S}^{d,q}_+$.
\end{rem}


Figure \ref{fig:GammaDistributions} illustrates Proposition \ref{prop_explicit_Gamma}. 
Each column contains the reference measure $U_{d,q}^+$ for $q = p/(p-1)$ in the first line, and the associated distribution with $i.i.d.$ components $X_i \sim \mathcal{L}_p$ below. 
Reference contours of orders $0.25, 0.5, 0.75$ are represented as well as the explicit MK quantile contours. 
Whereas the center-outward ordering intrinsic to $U_d$ is natural for elliptical models, \cite{chernozhukov2015mongekantorovich}[Section 2.4],
the left-to-right ordering may be more relevant under generalized Gamma models in $\RR^d_+$.

\begin{figure}[h]
\centering
\subfigure[$U_{d,q}^+$, $q=3$]{\includegraphics[width=1.55in,height=1.55in]{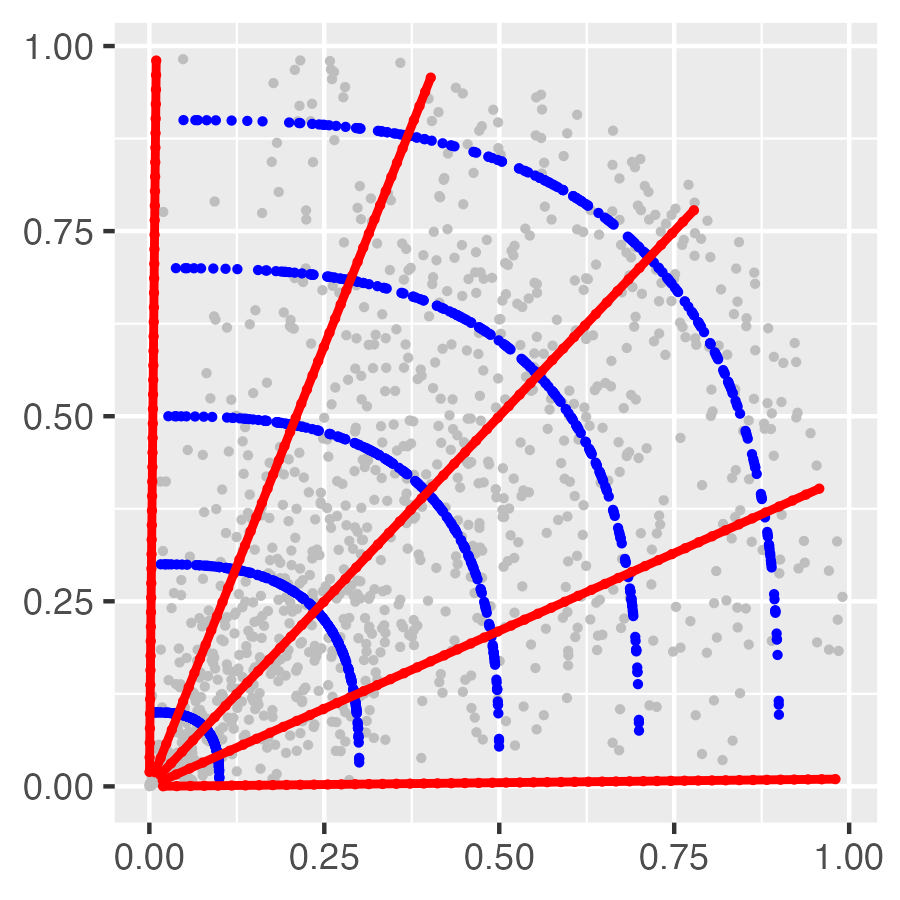}}
\subfigure[$U_{d,q}^+$, $q=2$]{\includegraphics[width=1.55in,height=1.55in]{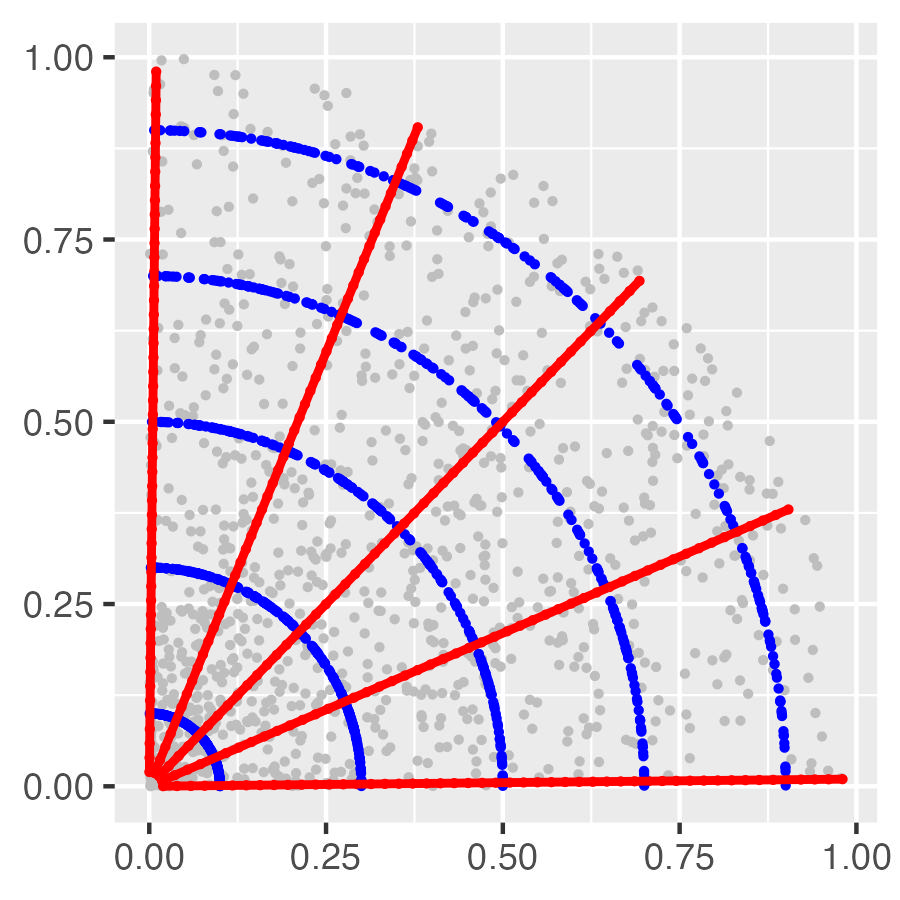}}
\subfigure[$U_{d,q}^+$, $q=5/4$]{\includegraphics[width=1.55in,height=1.55in]{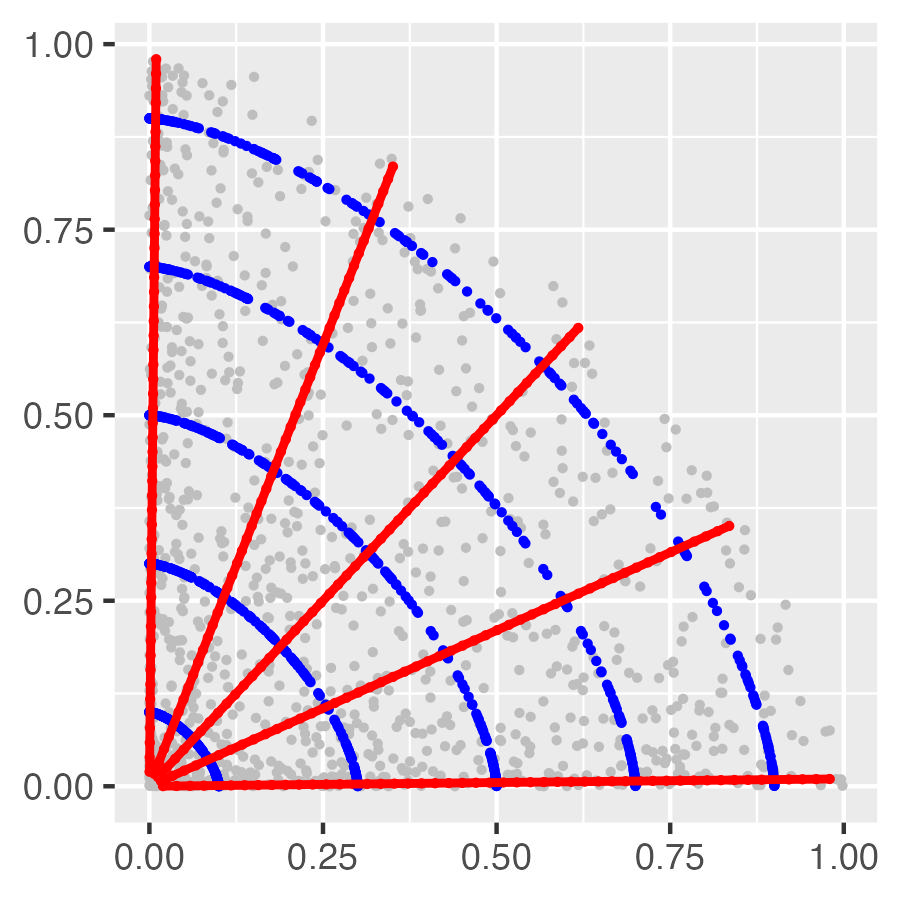}}

\subfigure[$p=1.5$]{\includegraphics[width=1.55in,height=1.55in]{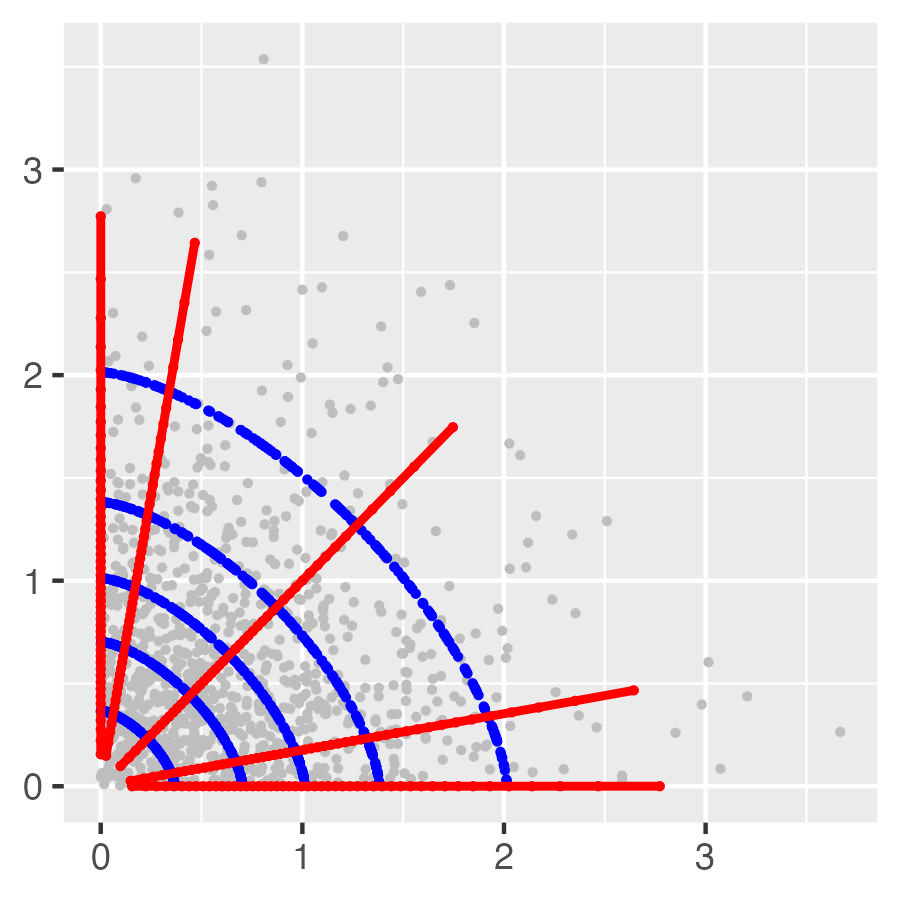}}
\subfigure[$p=2$]{\includegraphics[width=1.55in,height=1.55in]{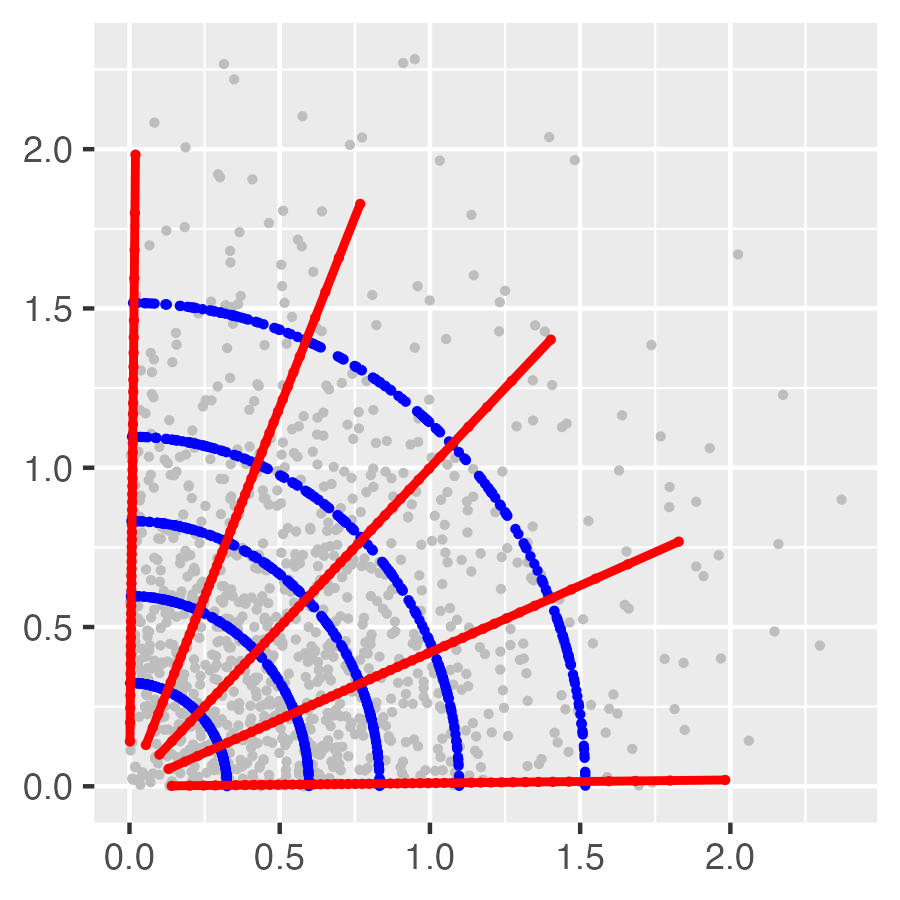}}
\subfigure[$p=5$]{\includegraphics[width=1.55in,height=1.55in]{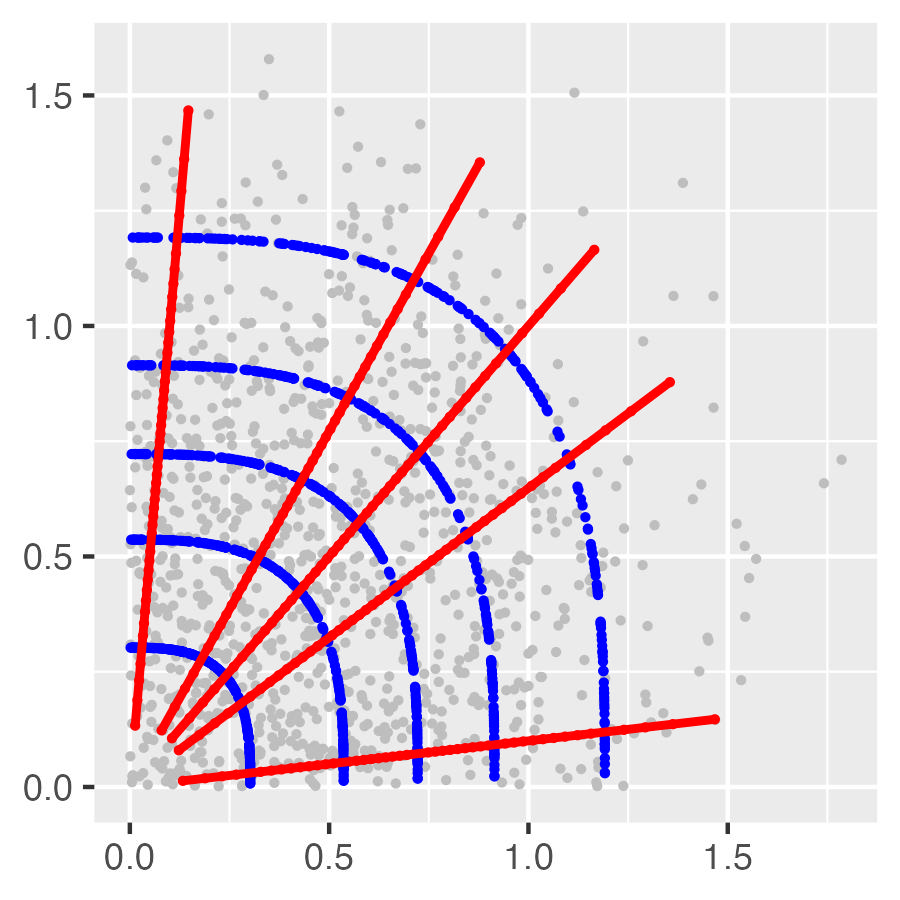}}

\caption{(First line) Reference measures $U_{d,q}^+$ for indices $q$ conjugate with $p$. (Second line) Generalized Gamma distributions $\mathcal{L}_p$ and explicit quantile contours and sign curves. }
\label{fig:GammaDistributions}
\end{figure}

\section{Multivariate values and vectors at risk}\label{sec:VVrisk}

Considering several univariate risks separately leads to neglecting the correlation structure between the components of a  random vector. Hence, studying a notion of risk in a multivariate way is of major importance.
We refer to \cite{Ekeland_2010}[Section 2.1] for motivations for intrinsically multivariate risk analysis problems.
The risk framework considers a vector of dependent losses $X\in\RR^d_+$, where each component is a positive measure whose unit may be different from one to another. Naturally, the larger the value of a component, the greater the associated risk.  
On the one hand, we argue that real-valued risk measures are not sufficient in such setting.
With real-valued random variables, real-valued risk measures are able to catch all the information needed. But, with the increase of the dimension, one may need vectors in $\RR^d$ to get the directional information of the multivariate tails. 
With this in mind, we make here a proposal for vector-valued risk measures, namely \textit{Vectors-at-Risk} and \textit{Conditional-Vectors-at-Risk}, which aim to sum up the relevant information contained in the center-outward quantiles and superquantiles.
Related to these vector-valued measures, we define \textit{multivariate values-at-risk} and \textit{multivariate conditional-values-at-risk} which are measures of the form of $\rho : \RR^d_+ \rightarrow \RR$ able to compare $X$ and $Y$ by $\rho(X)>\rho(Y)$. 
Importantly, the computation of these real-valued measures gives the vector-valued ones, hitting two targets with one shot without added complexity.
On the other hand, a multivariate extension of the concepts of VaR and CVaR shall have the same interpretation {  as} in $\RR$.
In dimension $d=1$, the VaR of some risk at a level $\alpha$ is simply the quantile of order $\alpha$. Accordingly, the CVaR is the superquantile of order $\alpha$. 
These are to be understood respectively as the worst risk encountered with $\nu$-probability $\alpha$, and as the averaged risk beyond this quantile. 
Being an observation from the underlying distribution, it shall be vector-valued in our multivariate framework.

By construction, the MK quantile contour of order $\alpha \in [0,1]$ contains the points having the most outward position with $\nu$-probability $\alpha$. 
Hence, we argue that a Vector-at-Risk of order $\alpha$ should belong to this contour. 
Even if it is always adapted to the geometry of $\nu$, a MK quantile contour can either describe a \textit{central} or a \textit{bottom-left} area, by changing the reference measure.
We explore both these possibilities, leaving it to the reader to decide which tool to adopt.
In what follows, we rely on $U_d$ and $U_d^+$, its restriction to $\RR^d_+$.
Of course, other $p$-spherical uniform distributions could be used, but we do not know if this would induce any benefit. 
Starting from the choice between \textit{central} or \textit{bottom-left} areas for MK quantile contours, the worst vectors of losses are the furthest from the origin in our context. Then, we suggest to select points from the center-outward quantile contour of order $\alpha$ with maximal norm, using Definition \ref{def:MKquantile}. 
\textit{Conditional-Vectors-at-Risk} (CVaRs) are defined in the same way, but considering the center-outward superquantiles, with Definition \ref{superq_contours}.

\begin{definition}\label{VaRsCVaRs}
The Vector-at-Risk at level $\alpha$ is defined by 
$$
\mbox{VaR}_{\alpha} (X) \in \argsupl \left\{ \; \Vert X \Vert_1 \; ; \; X \in \mathcal{C}_\alpha \; \right\}.
$$
Similarly, the Conditional-Vector-at-Risk is defined by 
$$
\mbox{CVaR}_{\alpha} (X) \in \argsupl \left\{ \; \Vert X \Vert_1 \; ; \; X \in \mathcal{C}_\alpha^s \; \right\}.
$$
\end{definition}

\noindent The choice of $\Vert \cdot \Vert_1$ in the above definition depends on how to compare different risks. 
We emphasize that this choice is very distinct to the one of the reference measure : this consideration takes place once the contours are fixed.
Without added information, we believe that two observations of same $1$-norm shall be considered of same importance. Intuitively, when comparing multivariate risks, the risks per component add up. For example, if 
\begin{equation*} 
x_1 = 
   \begin{pmatrix} 
         1 \\ 0
   \end{pmatrix} 
  \quad \text{and} \quad x_2 =
      \begin{pmatrix} 
         0.5 \\ 0.5
   \end{pmatrix} 
\end{equation*} 
belong to the same quantile contour of level $\alpha$, choosing $\Vert \cdot \Vert_2$ instead of $\Vert \cdot \Vert_1$ in Definition \ref{VaRsCVaRs} would induce to consider that $x_1$ is worst than $x_2$. 
This choice is coherent with the common practice of computing univariate VaR and CVaR on the random variable $\Vert X \Vert_1$ in financial applications, but a major difference is that it encodes the multivariate joint probability of the vector $X$ before applying the sum. Thus, it takes into account the correlations, while providing more information, because typical values for component-wise risks can be retrieved through our Vectors-at-Risk.
With this in mind, the Vector-at-Risk of order $\alpha$ is to be understood as the ``worst'' risk encountered with probability $\alpha$, whereas the Conditional-Vector-at-Risk is the average risk beyond this ``worst'' observation, where the notion of ``worst'' is characterized here by $\Vert \cdot \Vert_1$ and the focus is either on central or on bottom-left areas. Thus, this generalizes the understanding of univariate VaR and CVaR. 
We displayed in Figure \ref{fig:VaRCVaR_with_contours} the $\mbox{VaR}_{\alpha}$ and $\mbox{CVaR}_{\alpha}$ with quantile and superquantile contours estimated via the numerical scheme presented in Section \ref{choice_entropic}  for $\ee=10^{-3}$. Each red point is the worst observation inside the corresponding quantile or superquantile regions.

\begin{figure}[htbp]
\centering
\subfigure[Center-outward reference $U_d$]{\includegraphics[width=2in,height=2in]{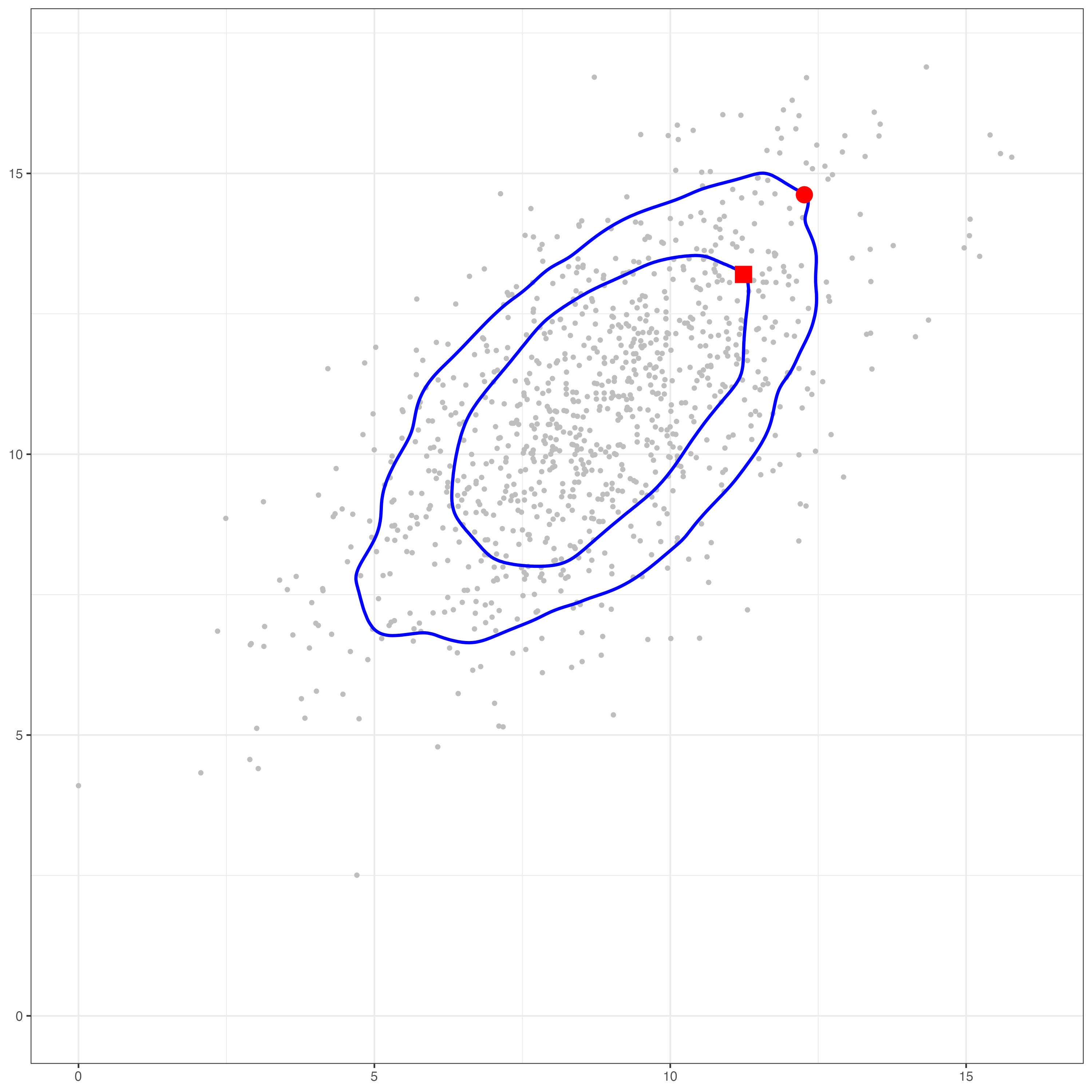}}
\subfigure[Left-to-right reference $U_d^+$]{\includegraphics[width=2in,height=2in]{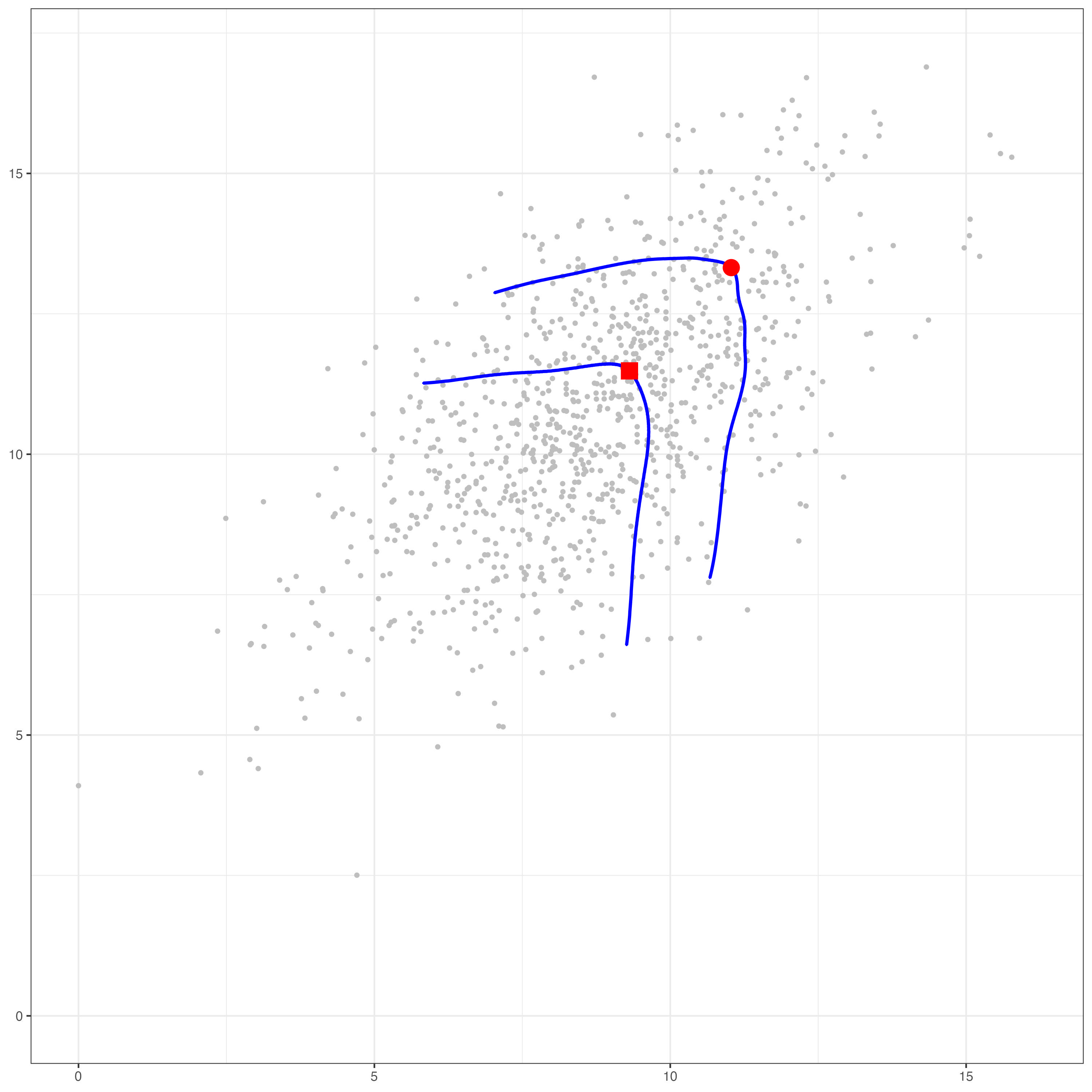}}

\caption{ {  Gaussian distribution with $\mbox{VaR}_{\alpha}$ and $\mbox{CVaR}_{\alpha}$ in red and associated MK quantile and superquantile contours in blue, for $\alpha = 0.5$.} }
\label{fig:VaRCVaR_with_contours}
\end{figure}

Intuitively, Definition \ref{VaRsCVaRs} already gives a real-valued risk measure. Note that the set $\left\{ \; \Vert X \Vert_1 \; ; \; X \in \mathcal{C}_\alpha \; \right\}$ is the same as $\left\{ \; \Vert Q_X(u) \Vert_1 \; ; \; \Vert u \Vert_2= \alpha \; \right\}$ {  for both the reference distribution $U_d$ and its restriction $U_d^+$ to $\RR^d_+$.}
\begin{definition}\label{rhoVaRsCVaRs}
The multivariate value-at-risk at level $\alpha$ is defined as 
$$
\rho_{\alpha}^Q (X) = \supl_u \left\{ \; \Vert Q_X(u) \Vert_1 \; ; \; \Vert u \Vert_2 = \alpha \; \right\}.
$$
Similarly, the multivariate conditional value-at-risk is defined as 
$$
\rho_{\alpha}^S (X) = \supl_u \left\{ \; \Vert S_X(u) \Vert_1 \; ; \; \Vert u \Vert_2 = \alpha \; \right\}.
$$
\end{definition}

These \textit{multivariate (conditional) values-at-risk} indicate on the spreadness of the point cloud drawn from a vector of losses, and contain information about its centrality and the correlations between its components.
Empirical experiments have been conducted in Section \ref{sec:emp} where they are compared with the \textit{maximal correlation risk measure} from \cite{beirlant2019centeroutward}.

\section{Numerical experiments.} \label{sec:numexp}

\subsection{On the choice of the empirical quantile map.}\label{choice_entropic}

Hereafter, we chose to use the entropic map as a regularized empirical map, as advocated in \cite{BBT2023,Carlier:2022wq,Masud2021}. 
Regularizing the quantile function enables to reach a smooth function $T:u \mapsto T(u)$ in practice, 
and has been the concern of several works, including \cite{beirlant2019centeroutward,Hallin-AOS_2021}.
Our regularization choice is thus motivated by 
the known computational advantages of entropic optimal transport, \cite{cuturi2013sinkhorn,aude2016stochastic}, with a very recent line of work focusing on the entropic map \cite{goldfeld2022limit,Pool2023,pooladian2021entropic,RigolletStromme2022}. 
A relaxed version of Monge problem \eqref{MongeOT} is the Kantorovich problem, see the introduction of \cite{Villani}. 
This allows for easier computations, even more so when one uses entropically regularized optimal transport. We only state here the semi-dual formulation of Kantorovich regularized problem, that one can find for instance in \cite{aude2016stochastic}. 
The notation differ from traditional ones in optimal transport theory, to be consistent with the context of center-outward quantiles. 
Let $\mathcal C(\XX)$ be the space of continuous functions from $\XX$ to $\RR$. Let $\ee>0$ be the regularization parameter, which must be low in order to approximate better the true OT. For $\mu,\nu \in \mathcal{P}_2(\RR^d)$ with finite second-order moments and with respective supports $\mathcal{U},\XX$, and for a given $x_0\in \XX$, the semi-dual problem has a unique solution and writes
\begin{equation}
\max_{v \in \mathcal{C}(\mathcal{X}) : v(x_0)=0} \int_\mathcal{U} v^{c,\epsilon}(u) d\mu(u) + \int_\mathcal{X} v(x) d\nu(x) - \epsilon, 
\label{Se}
\end{equation}
where $v^{c,\epsilon} \in \mathcal C(\mathcal{U})$ is the smooth c-transform of $v$ 
\begin{equation}\label{def_c_transf}
 \forall u \in \mathcal{U}, v^{c,\epsilon}(u) = -\epsilon \log \Big( \int_{\XX} \exp \Big( \frac{v(x) -\frac{1}{2}\Vert u-x\Vert^2}{\epsilon} \Big) d\nu(x) \Big) .
\end{equation}
In practice, $\nu$, the law {  whose} we consider the quantiles, is not known and is approached by its empirical counterpart, so it always has a finite moment of order $2$. On the contrary, $\mu$, the reference law, is known and absolutely continuous. 
In this semi-discrete setting, one can use stochastic algorithms \cite{bercu2020asymptotic,BBT2023,aude2016stochastic} to solve \eqref{Se} and obtain the optimal Kantorovich potential $v$. 
From this, one can deduce the entropic map, \cite{pooladian2021entropic}, an approximation of the Monge map partially legitimated by an entropic analog of Brenier's theorem, \cite{Brenier1991PolarFA,Cuesta1989NotesOT} 
$$
Q_\ee(u) = \nabla \Big( \frac{1}{2}\Vert u \Vert^2 - v^{c,\epsilon}(u) \Big).
$$
Additionally, neither Assumption \ref{hypA} nor Assumption \ref{hypB} are necessary to its definition or to ensure its continuity. 
Using \eqref{def_c_transf}, one has the analytic formula 
\begin{equation}
Q_{\epsilon}(u) = \int_{\XX} x g_\epsilon(u,x) d\nu(x),
\label{Def_Qe}
\end{equation}
where
\begin{equation*}\label{def_g_e}
g_\epsilon(u,x) = \frac{ \exp \left( \frac{ v_\epsilon(x) - \frac{1}{2}\Vert u-x \Vert^2 }{\epsilon} \right)} { \int_{\XX} \exp \left( \frac{ v_\epsilon(z) - \frac{1}{2}\Vert u-z \Vert^2 }{\epsilon}  \right) d\nu(z) } = \exp \Big( \frac{v_\epsilon^{c,\epsilon}(u) + v_\epsilon(x) - \frac{1}{2}\Vert u-x \Vert^2}{\epsilon} \Big).
\end{equation*}
This formula can be read as a conditional expectation of $\nu$ w.r.t.\ the conditional law associated with the transport plan $d\pi_\ee(u,x)=g_\epsilon(u,x) d\nu(x)d\mu(u)$, which appears to be solution of the Kantorovich regularized problem, \cite{aude2016stochastic}.
Plugging $Q_\ee$ into Definitions \ref{superquantile} and \ref{expectedshortfall} induces the entropic analogs
\begin{equation}\label{Def_Se}
S_{\epsilon}(u) = \frac{1}{1-\Vert u \Vert} \int_{\Vert u \Vert}^1 Q_\ee\Big(t\frac{u}{\Vert u \Vert}\Big) dt 
\end{equation}
and
\begin{equation}
E_{\epsilon}(u) = \frac{1}{\Vert u \Vert} \int_0^{\Vert u \Vert} Q_\ee\Big(t\frac{u}{\Vert u \Vert}\Big) dt.
\end{equation}

The empirical counterparts $\widehat{Q}_{\ee,n}$, $ \widehat{S}_{\ee,n}$ and $ \widehat{E}_{\ee,n}$ of definitions \eqref{Def_Qe} and \eqref{Def_Se} are obtained by plug-in estimators of the problem \eqref{Se} for the empirical measure $\widehat{\nu}_n = \frac{1}{n} \sum_{i=1}^n \delta_{X_i}$ and {  for the reference distribution $\mu$}.
In the experiments of the next section, we use the stochastic Robbins-Monro algorithm taken from \cite{bercu2020asymptotic}. 
The entropic map $\widehat{Q}_{\ee,n}$ defined in \eqref{Def_Qe} is considered as a regularized empirical quantile function with $\ee = 10^{-3}$, and the integrals in \eqref{Def_Se} are estimated by Riemann sums. 
{   Because the} entropic map is continuous and it is the gradient of a convex function, even empirically, it shall provide nested and smooth contours. 
Moreover, estimating Monge maps is known to suffer from the curse of dimensionality, see \cite{hutter2020minimax,pooladian2021entropic}. 
But, when the objective is the entropic map, for any $\ee>0$, convergence rates independent from $d$ were obtained in \cite{RigolletStromme2022}. 
Furthermore, the empirical version of the entropic map approaches the true Monge map $(\ee = 0)$ with a rate independent from $d$ if $\nu$ is assumed to be discrete, see \cite{Pool2023}. 
Note that, even for low $\ee$, the entropic map does not exactly push-forward {  the reference distribution} onto $\nu$, and is rather an approximation of it.

\subsection{Empirical study on simulated data}\label{sec:emp}
\subsubsection{Descriptive plots for quantiles and superquantiles}



Descriptive plots associated with Definition \ref{superq_contours}
are given in Figure \ref{fig:descriptiveplots}, to describe all the information contained in our new {  concepts}.
The continuous reference distribution $U_d$ is transported to the discrete banana-shaped measure $\nu$ with support of size $n = 5000$, via our empirical center-outward expected shortfall (resp.\ superquantile) function. 
Note how the data splits into a central area and a periphery area by the range of points covered by both maps. 
These are satisfying estimators for the well-suited $\ee = 10^{-3}$, that is to say the first column of Figure \ref{fig:descriptiveplots}.
As the regularization parameter $\ee$ for $E_\ee$ and $S_\ee$ grows, the contours concentrate around the mean vector of $\nu$, which is a known feature of entropic optimal transport. 
One can observe that our regularized approach yields smooth interpolation between image points. For visualization purposes, the red points of averaged sign curves are linked by straight paths. The blue points are not, to illustrate how the contours capture the empty space in the middle of the non convex point cloud.

\begin{figure}[h]
\centering
\subfigure[Expected shortfalls, $\ee = 0.001$]{\includegraphics[width=1.5in,height=1.5in]{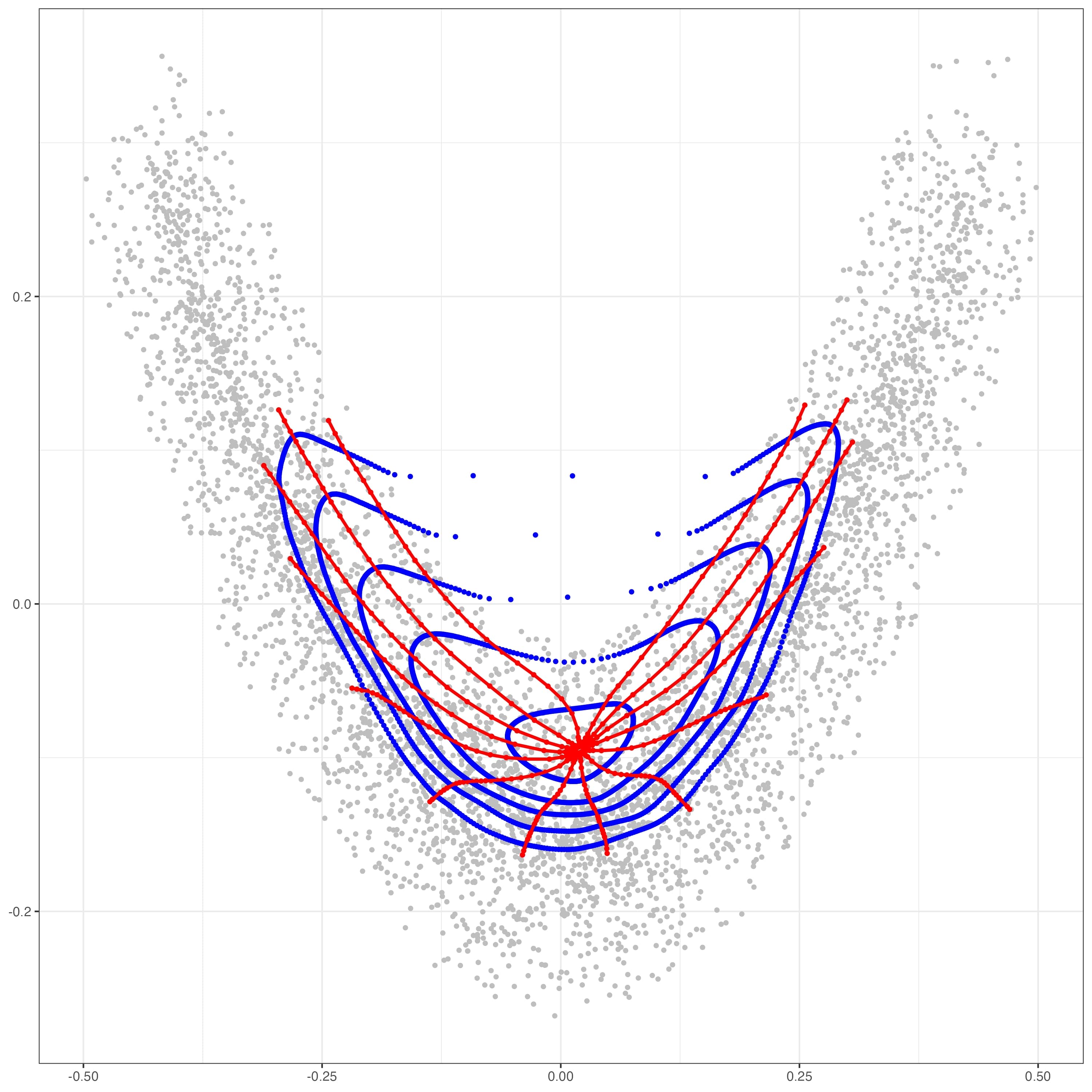}}
\subfigure[Expected shortfalls, $\ee = 0.005$]{\includegraphics[width=1.5in,height=1.5in]{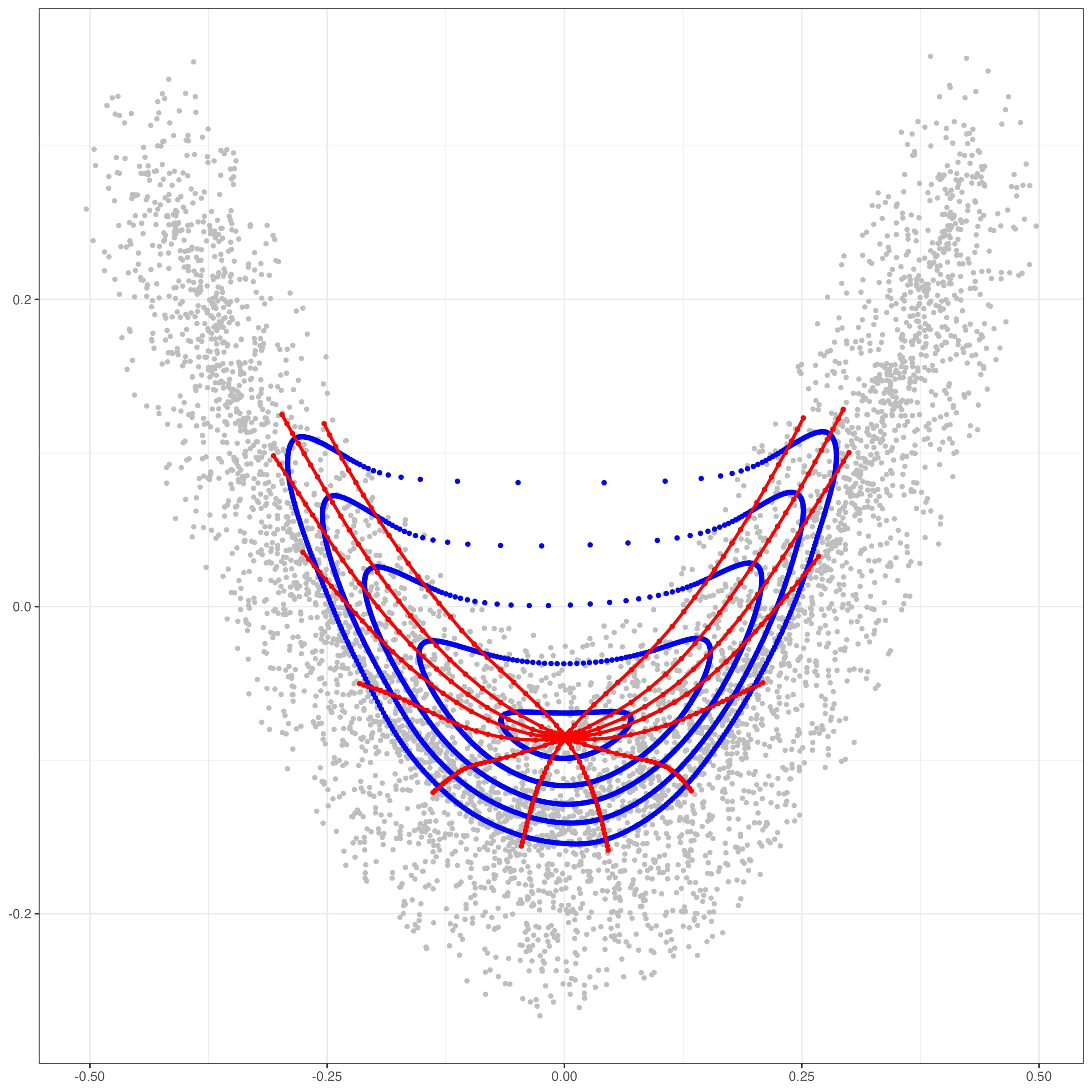}} 
\subfigure[Expected shortfalls, $\ee = 0.01$]{\includegraphics[width=1.5in,height=1.5in]{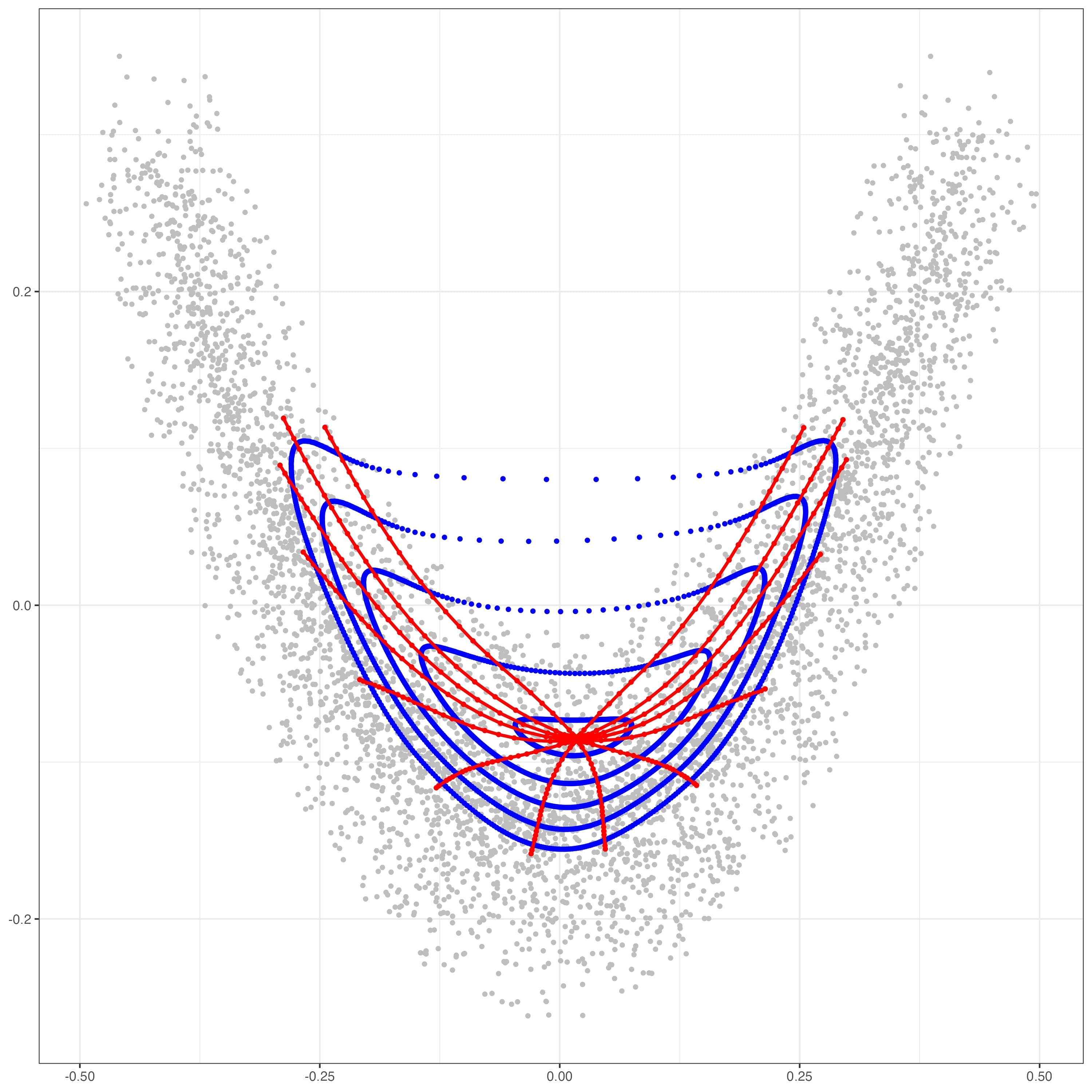}}

\subfigure[Superquantiles, $\ee = 0.001$]{\includegraphics[width=1.5in,height=1.5in]{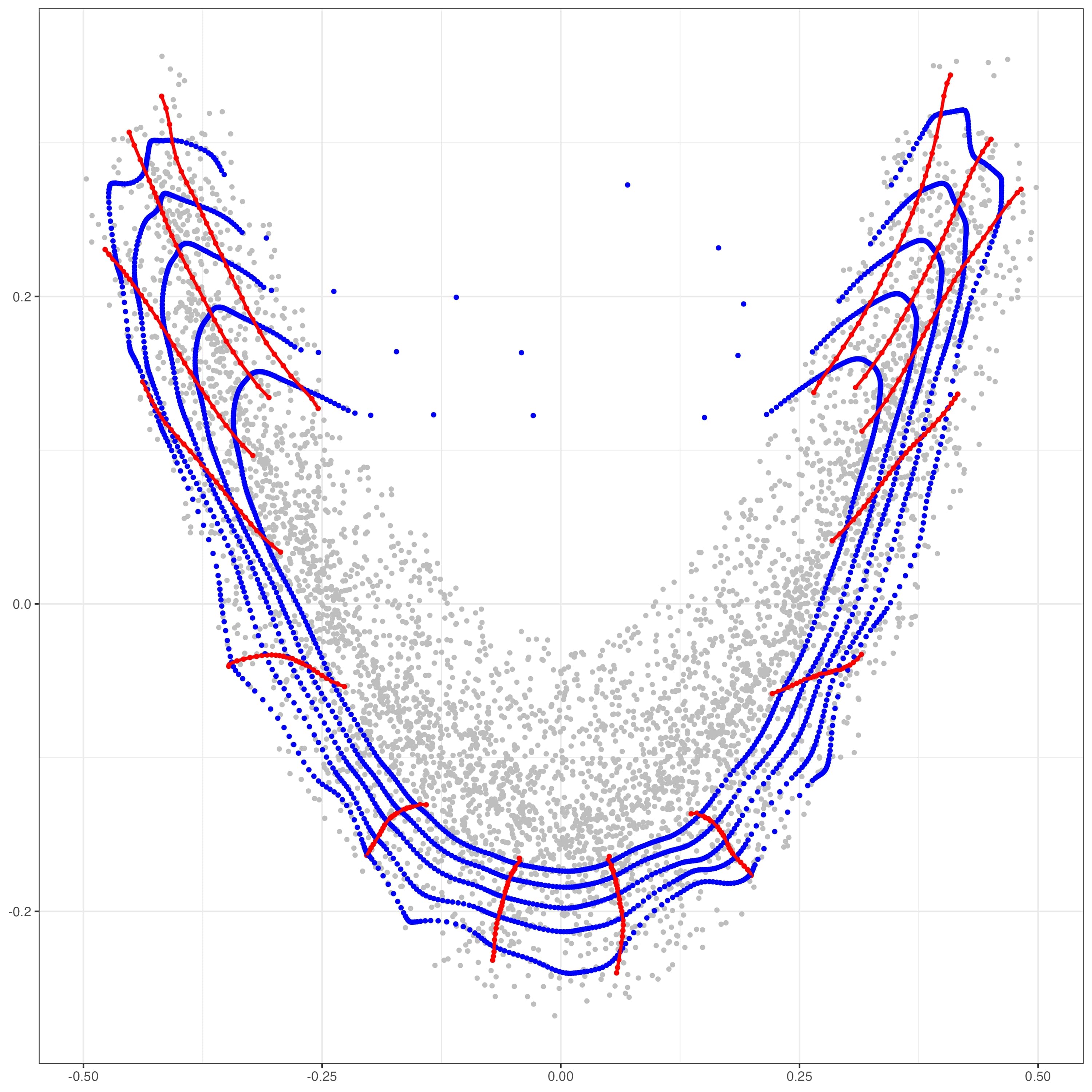}}
\subfigure[Superquantiles, $\ee = 0.005$]{\includegraphics[width=1.5in,height=1.5in]{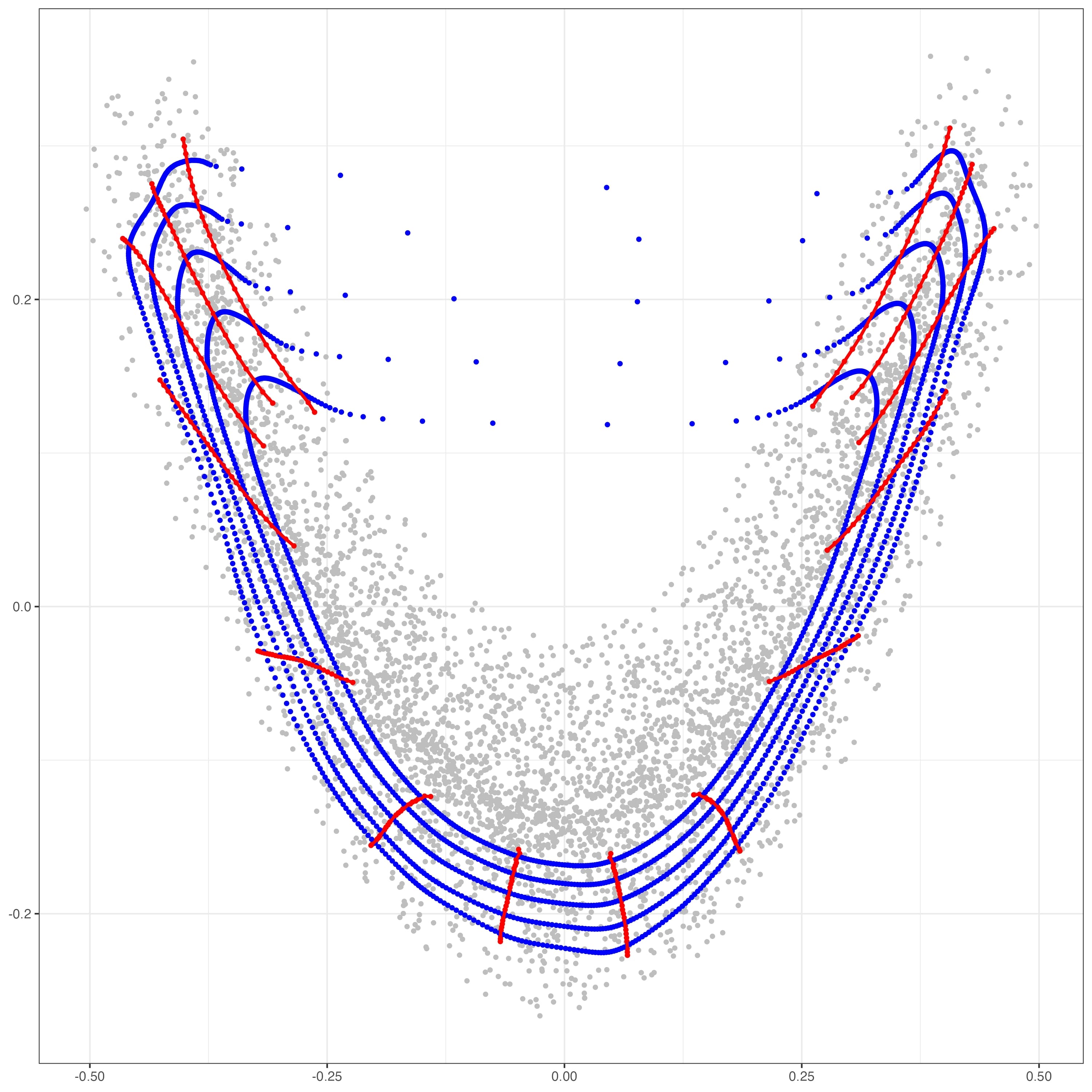}} 
\subfigure[Superquantiles, $\ee = 0.01$]{\includegraphics[width=1.5in,height=1.5in]{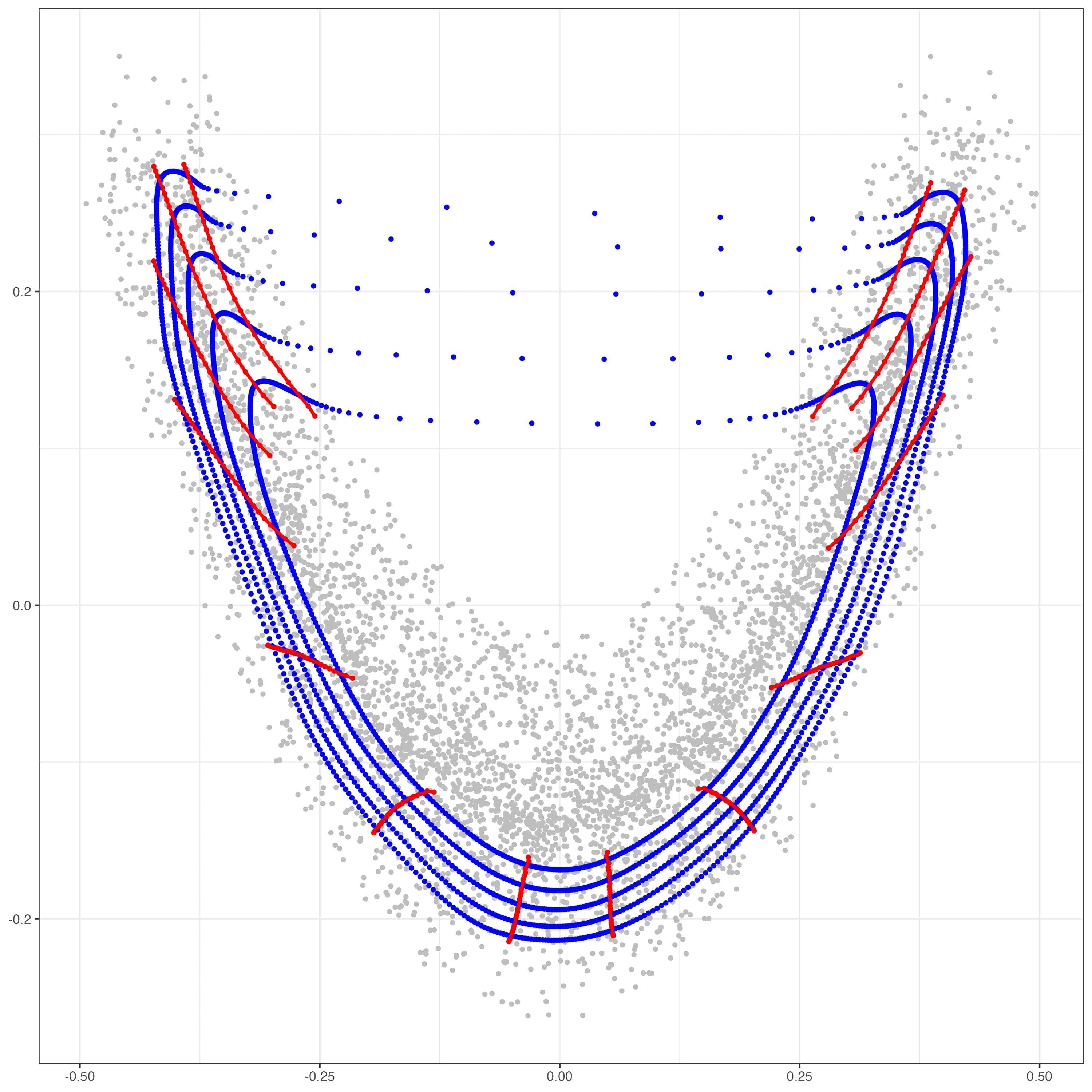}}

\caption{{  Center-outward} expected shortfall and superquantile contours of levels $\alpha$ in $\{0.1,0.3,0.5,0.7,0.9\}$ in blue and averaged sign curves $C_u$ in red.}
\label{fig:descriptiveplots}
\end{figure} 

\subsubsection{Risk measurements on toy examples}\label{subsec:emp_descriptif}

 \begin{figure}[htbp]
\centering

{\subfigure[Same {  Gaussians}, reduced covariance matrix. ]{\includegraphics[scale=0.22]{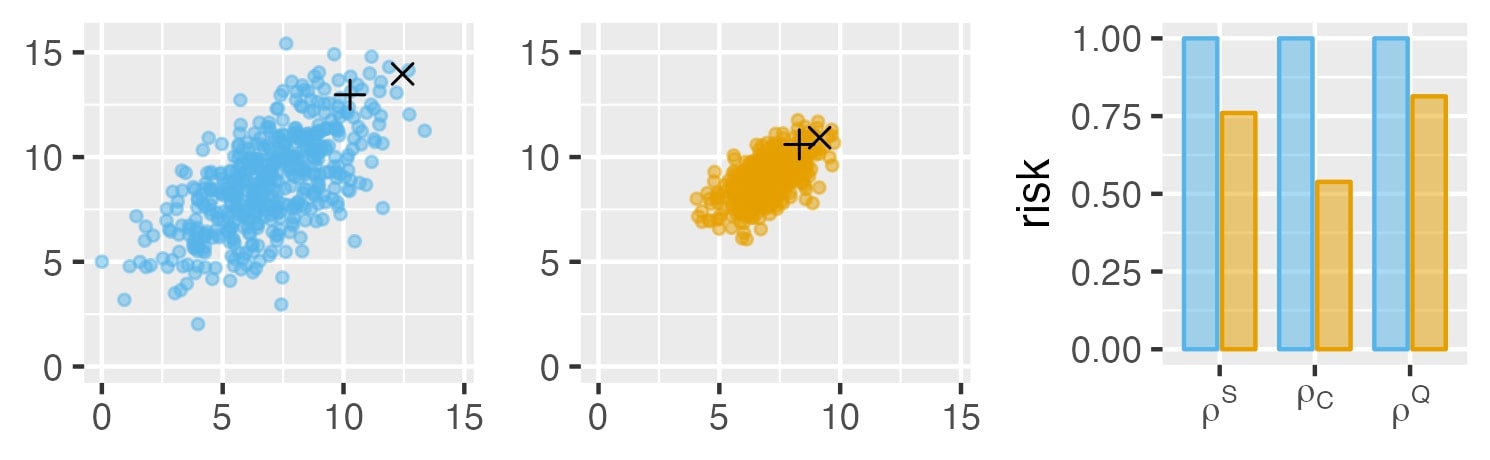}\label{fig:toy_ex_1}}}

{\subfigure[Same {  Gaussians}, added outliers.]{\includegraphics[scale=0.22]{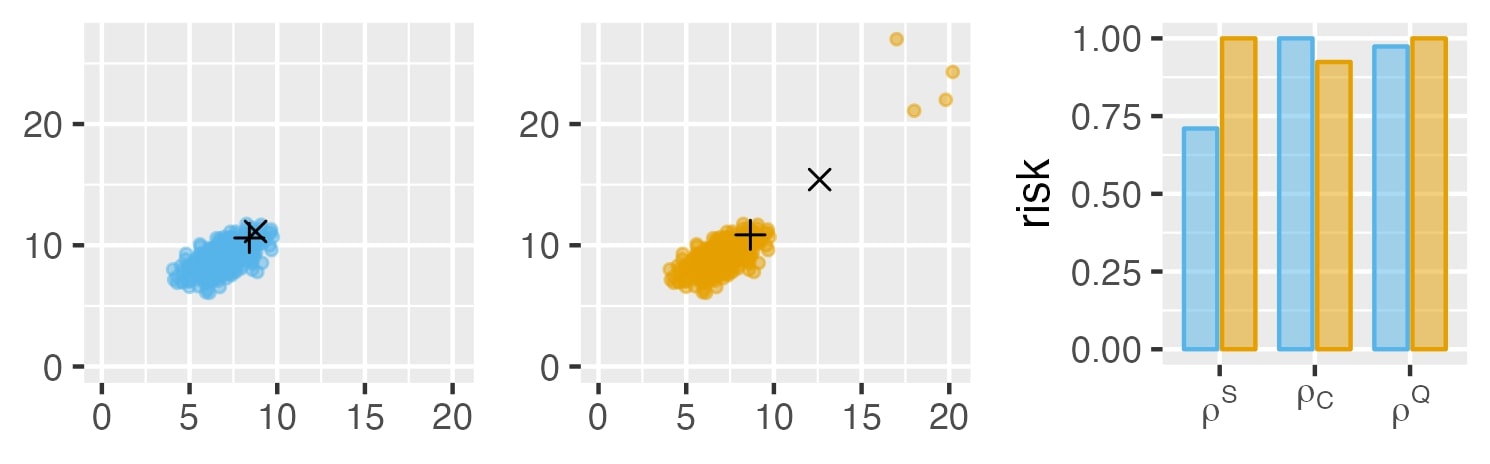}\label{fig:toy_ex_2}}}

{\subfigure[Same {  Gaussians} with shift.]{\includegraphics[scale=0.22]{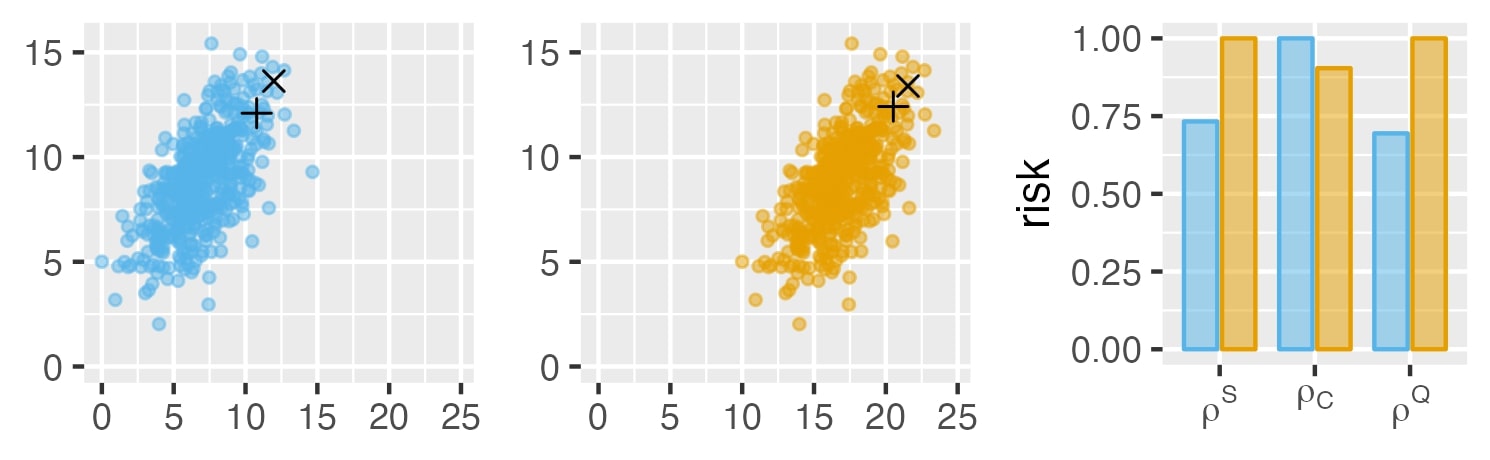}\label{fig:toy_ex_3}}}


{\subfigure[{  Gaussians pointed in the vertical or horizontal direction }]{\includegraphics[scale=0.22]{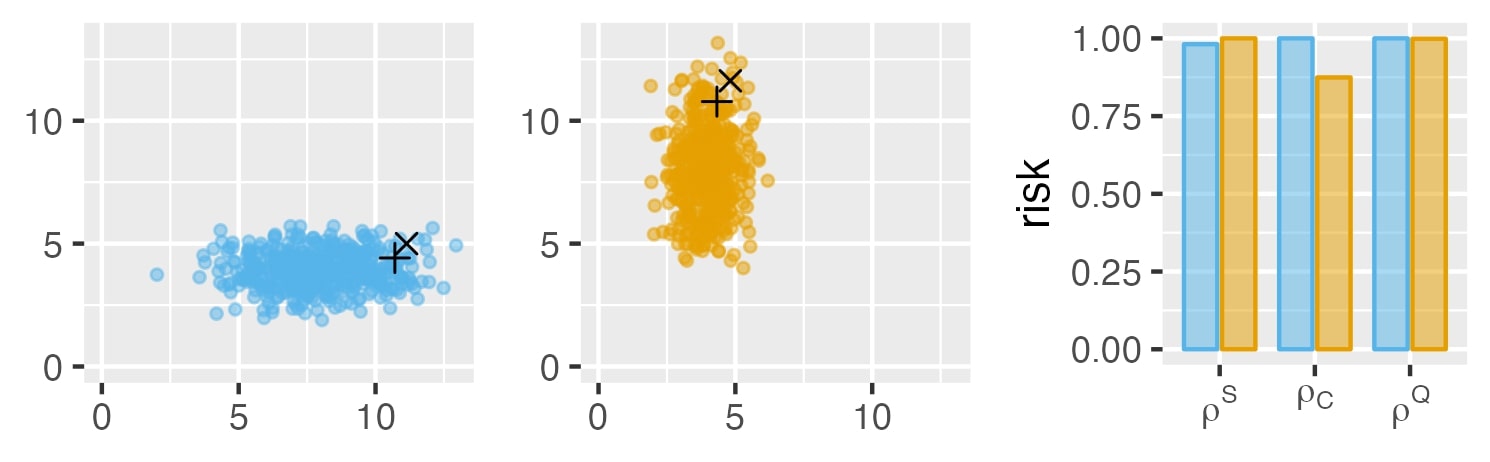}\label{fig:toy_ex_5}}}

\caption{VaRs ($+$) and CVaRs ($\times$) on toy examples. Third column : real-valued risk measurements for the point clouds of the same line.}
\label{fig:easy-visual-risks-var-cvar}
\end{figure}

In dimension $d=2$, an empirical distribution can be represented with a scatter plot and the riskiest observations are visible with the naked eye: they are located furthest from the origin.
From this principle, we selected easy-to-handle situations to evaluate our risk measures in Figure \ref{fig:easy-visual-risks-var-cvar}.
Each row refers to a situation where a blue scatter plot (first column) is to be compared with an orange one (second column). For each situation, our \textit{(Conditional) Vectors-at-Risk} of order $\alpha=0.75$ is illustrated on each scatter plot. 
Moreover, the associated $\rho_{\alpha}^Q$ and $\rho_{\alpha}^S$ are computed and compared to the maximal correlation risk mesure $\rho_C$ from \cite{beirlant2019centeroutward} in the third column.
The higher the bar, the riskier the corresponding vector of losses. 
In view of their comparison, the measurements $\rho \in \{\rho_{\alpha}^Q,\rho_{\alpha}^S, \rho_C\}$ are rescaled. For $Y_1$ the distribution of the blue scatter plot and $Y_2$ the orange one, one considers, for the height of the bars, 
\begin{equation*}
\rho(Y_1)/ \max\Big(\rho(Y_1),\rho(Y_2)\Big) \quad \text{and} \quad \rho(Y_2)/ \max\Big(\rho(Y_1),\rho(Y_2)\Big).
\end{equation*}
For each situation, our VaR and CVaR provide typical observations in the multivariate tails.
The \textit{maximal-correlation risk measure} $\rho_C$ performs as well as expected, while our CVaR succeeds in more situations. Indeed, $\rho_C$ benefits from several theoretical properties but only measures the riskiness of $X-\mathbb{E}(X)$, hence it neglects the shift effects.
 
Figure \ref{fig:toy_ex_1} contains two Gaussian distributions with identical mean vectors and covariance matrices related through the multiplication by a positive real. 
This is well tackled by each real-valued risk measurement. 
The existence of more outliers must be taken into account similarly, as in figure \ref{fig:toy_ex_2}, where the two scatter plots originate from the same underlying distribution, but some outliers are added to the orange one. The CVaR is much more sensitive to these outliers than the VaR, as in dimension $d=1$.  Note that $\rho_C$ ignores the outliers and leads to the wrong decision in the sense than the blue distribution is considered as the riskiest one.
In the situation of Figure \ref{fig:toy_ex_3}, the orange scatterplot is identical to the blue one, but shifted to the right side, so that it must be the riskiest. As $\rho_C$ ignores this shift, it induces the wrong decision.
Finally, in the last example of Figure \ref{fig:toy_ex_5},
making a decision on the relative risk between the underlying distributions requires a preference for one of the components. 
Here, the two situations reveal risks of same intensity but directed towards different directions.
To inform on the underlying directional information, vector-valued risk measures such as our VaRs and CVaRs are needed in complement. 

\subsection{Risk measurements on wind gusts data}\label{subsec:windgusts}

In this section, we illustrate our new multivariate risk measures on the analysis of a real dataset provided by the \href{https://cran.r-project.org/web/packages/ExtremalDep/ExtremalDep.pdf}{ExtremalDep} R package, \cite{ExtremalDep}, and dedicated to the study of strong wind gusts. 
This dataset has previously been studied in \cite{DiBernardino2018,Goegebeur2023,Marcon2017} in a context of risk measurement.
The three variables are hourly wind gust (WG) in meters per second, wind speed (WS) in meters per second, and air pressure at sea level (DP) in millibars, recorded at Parcay-Meslay (France) between July 2004 and July 2013. We consider the 1450 weekly maximum of each measurement. 
  
Because the variables are of different nature, it is the precise framework where multivariate risk analysis is useful, rather than the aggregation of several variables. 

\begin{figure}[h]
\centering
\includegraphics[height=2.8in]{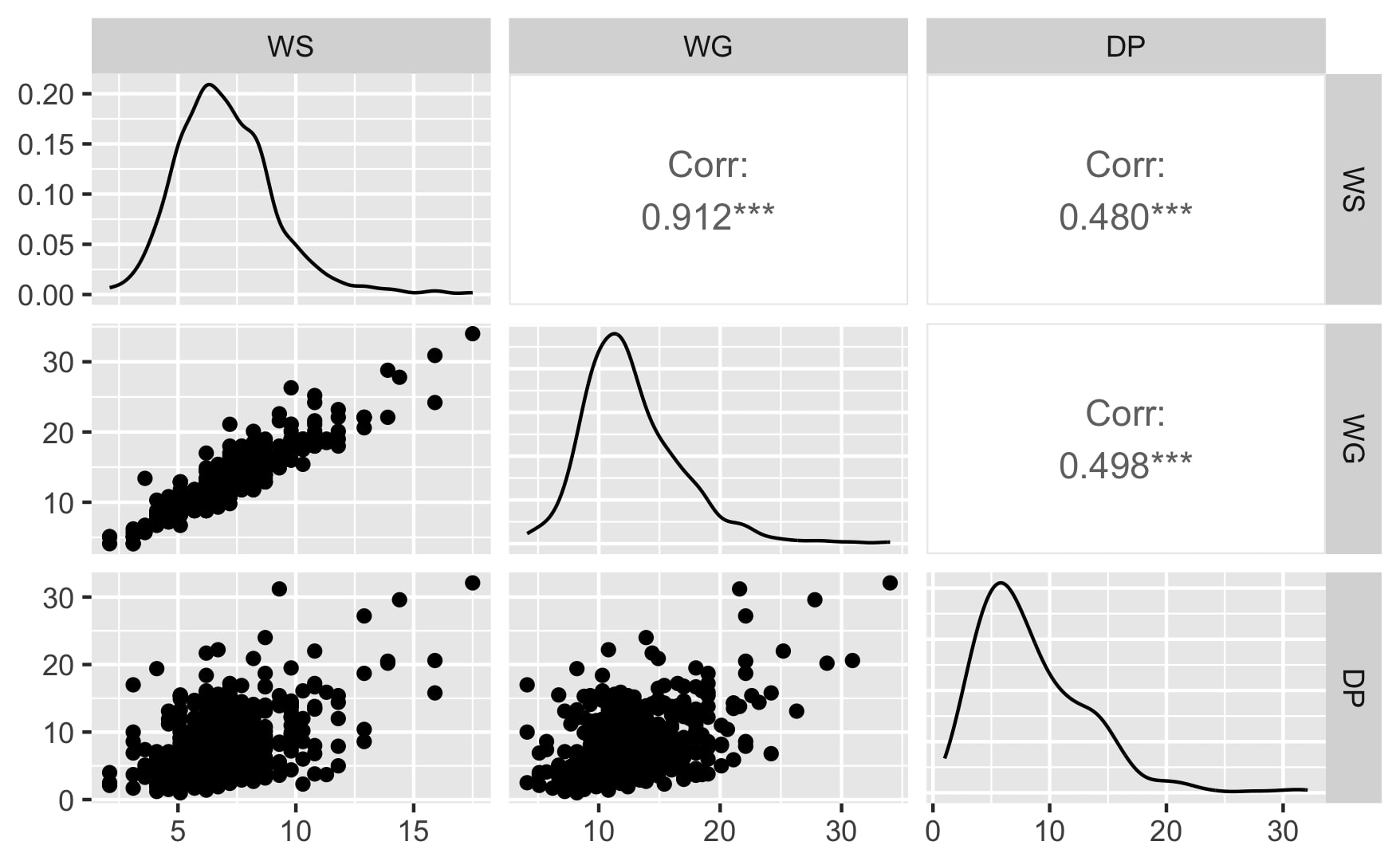}
\caption{Three dimensional wind gust data set.}
\label{fig:corplot}
\end{figure} 

Figure \ref{fig:corplot} represents our three-dimensional dataset with pair scatterplots under the diagonal and Pearson correlation values above. 
The diagonal represents empirical density functions of each variable. Upper-right dependence can be observed and has physical explanations. Strong wind gusts occur with stormy weather, during which strong wind speed and high air pressure are frequently recorded. 

\begin{figure}[h]
\centering
\subfigure[Risk level 0.25]{\includegraphics[height=1.5in]{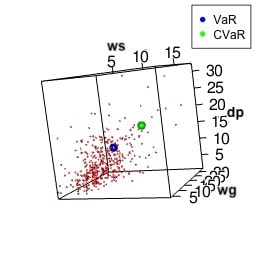}}
\subfigure[Risk level 0.5]{\includegraphics[height=1.5in]{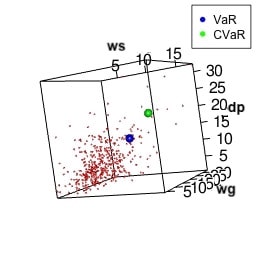}}
\subfigure[Risk level 0.75]{\includegraphics[height=1.5in]{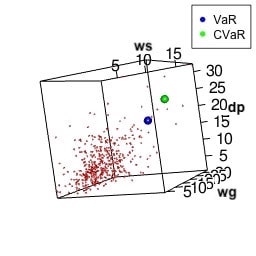}}
\caption{Vectors-at-Risk and Conditional-Vectors-at-Risk.}\label{fig:plot3d}
\end{figure}

Figure \ref{fig:plot3d} represents the three dimensional empirical distribution in red together with our vectorial risk measures.
With the increase of the dimension, such representative plots are no longer convenient.
Rather, these measurements can be retrieved as in Table \ref{fig:tables}. 
\vspace{2ex}

\begin{table}[h!]
\centering
\begin{tabular}{||c c c c||} 
 \hline
 \ & $VaR_{0.25}$ & $VaR_{0.5}$ & $VaR_{0.75}$ \\ [0.5ex] 
 \hline
 WS & 8.21 & 9.98 & 12.36 \\ 
 WG & 15.06 &  \cellcolor{lightgray} 18.45 & 21.58 \\
 DP & 11.92 & 13.65 & 17.84 \\ [1ex] 
 \hline
\end{tabular}
\begin{tabular}{||c c c c||} 
 \hline
\ & $CVaR_{0.25}$ & $CVaR_{0.5}$ & $CVaR_{0.75}$ \\ [0.5ex] 
 \hline
 WS & 11.40 & 12.37 & 14.26 \\ 
 WG & 21.28 & 22.43 & 26.24 \\
 DP & 16.68 & 19.86 & 23.15 \\ [1ex] 
 \hline
\end{tabular}
\caption{Components of (Conditional) Vectors-at-Risk for several risk levels.}
\label{fig:tables}
\end{table}

This summarizes the targeted information contained in the dataset. For instance, with a given probability 0.25, 0.5 or 0.75, one shall expect, at worst, wind gusts of respective speed 
15.06, 18.45, 21.58 m/s. With respectively same probability, averaged observations beyond these worst cases shall lie around 21.28, 22.43, 26.24 m/s. 
  
As in \cite{Marcon2017}, we interpret these results thanks to the Beaufort scale. Note that this scale does not capture the speed of wind gusts, as it usually averages over 10 minutes, by convention. 
Values between 13.9 and 17.1 m/s can be considered as high winds. Strong winds begin with 17.2 m/s, with severely strong winds above 20.7 m/s up to 24.4 m/s. Severely strong winds can cause slight structural damage, but less than storms for values around 24.5-28.4 m/s. Above and up to 32.6 m/s, violent storms are very rarely experienced and cause widespread damage. Wind speeds greater than 32.7 m/s correspond to hurricanes.
Thanks to Table \ref{fig:tables}, with probability 0.75, the worst scenarios in Parcay-Meslay for wind gusts are severely strong winds. Moreover, along the tail events corresponding to 25\% of occurrences, one shall expect wind gusts of same speed as storms.
Also, to illustrate the fact that considering univariate measures leads to underestimating the risk, we display in Figure \ref{fig:table1D} traditional univariate Values-at-Risk of WG, the wind gusts. For example, its median is 11.80, (strong breeze), and must be compared with 18.45, (strong winds), the second coordinate of $VaR_{0.5}$. 

\begin{table}[h!]
\centering
\begin{tabular}{||c c c c c||} 
 \hline
 Min & 1st quartile & Median & 3rd quartile & Max \\ [0.5ex] 
 \hline
 4.10 & 9.80 & \cellcolor{lightgray} 11.80 & 14.90 & 34 \\ [1ex] 
 \hline
\end{tabular}
\caption{Univariate quantiles of our variable WG.}
\label{fig:table1D}
\end{table}

Such differences have statistical foundings. 
The median depends on the \textit{univariate} empirical distribution of WG
to describe a probability of $1/2$. 
Conversely, our multivariate $VaR_{0.5}$ encodes the \textit{multivariate} joint probability of the whole point cloud (WG,WS,DP). 
This can be summarized by the fact that univariate risk measures \textit{neglect the correlations}, which legitimates the use of multivariate risk analysis.

\section{Conclusion and perspectives} \label{sec:conclu}

In this paper, we provided new concepts of superquantiles and expected shortfalls, in a meaningful multivariate way, based on the Monge-Kantorovich quantiles. 
We expect that our definitions allow the extension of the univariate applications to the multivariate case. 
Also, we introduced new transport-based risk measures, extending the concepts of Value-at-Risk and Conditional-Value-at-Risk to the dimension $d>1$. 
The definitions of this paper are mainly motivated by practical concerns and the coherence of the proposed concepts with respect to their interpretation in dimension $d=1$. 
Interestingly, from our center-outward expected shortfall function, one can retrieve Lorenz functions defined in \cite{Hallin_mordant_2022}. 
This is discussed in Section \ref{integratedQuantiles}, and other perspectives are also presented in Section \ref{perspectives}.

\subsection{On the class of integrated quantile functions}\label{integratedQuantiles}

The role of integrated quantile functions in dimension $d=1$ has been highlighted in \cite{Guschin_2017}.
Obviously, we belong to the line of works trying to extend such functions to the setting $d>1$.
In Section 4 of \cite{Hallin_mordant_2022}, two different approaches are discussed, in view of generalizing 
\begin{equation}\label{ell1d}
\alpha \mapsto \int_0^\alpha Q(t) dt.
\end{equation}
Hereafter, we bridge their concepts with ours, namely Definition \ref{expectedshortfall} that extends \eqref{ell1d}.
Recall that $\mathbb{S}^{d-1} = \{\varphi \in \RR^d : \Vert \varphi \Vert_2 = 1 \}$ and $\P_S$ is the uniform probability measure on $\mathbb{S}^{d-1} $.
There, the (absolute) center-outward Lorenz function from \cite{Hallin_mordant_2022}[Definition 4] writes as
\begin{equation}\label{LorenzAndExpected}
L_{\textbf{X}\pm} : \alpha \mapsto \mathbb{E}[X \mathds{1}_{ X \in \mathds{C}_\alpha} ]= \int_{\mathbb{S}^{d-1}} \alpha E(\alpha \varphi) d\P_{S}(\varphi).
\end{equation}
This being said, we believe that our proposed concepts of integration along sign curves contain more information, namely directional. When characterizing 
the contributions of central regions to the expectation, $L_{\textbf{X}\pm}$ provides meaningful concepts of Lorenz curves, but this approach is insufficient for superquantiles and multivariate tails, as illustrated in Example \ref{example1}. 

Another important generalization of \eqref{ell1d} is the one from \cite{Fan2022} about multivariate Lorenz curves.
A main difference between the concepts of \cite{Fan2022} and \cite{Hallin_mordant_2022} is the reference distribution, either the uniform on the unit hypercube or the spherical uniform. 
Here, Lorenz curves aim to visualize inequalities within a given population. In this context, one could focus either on the contribution of middle classes as in \cite{Hallin_mordant_2022}, or on the one of the bottom of the population as in \cite{Fan2022}, in a way that these works are in fact complementary.
Furthermore, another approach is proposed in \cite{Hallin_mordant_2022}[Definition 6], that is real-valued, namely the (absolute) center-outward Lorenz potential function, for $\varphi \in \mathbb{S}^{d-1}$,
\begin{equation}\label{LorenzPotentialAndExpected}
\alpha \mapsto \EE_\varphi[ \psi(\alpha \varphi)].
\end{equation}
This is quite natural because, at the core of the univariate considerations of \cite{Guschin_2017}, one can retrieve the property that the quantile function is the gradient of a convex potential, that is the definition of the center-outward quantile function $Q$.
Also, it has a physical interpretation, with a measurement of the \textit{work} of the quantile function. There, one might note that our center-outward expected shortfall function draws a connection between \cite{Hallin_mordant_2022}[Definition 4] and \cite{Hallin_mordant_2022}[Definition 6], through \eqref{LorenzAndExpected} and \eqref{eq:psi_is_E}. Somehow, this strengthens the relation between $L_{\textbf{X}\pm}$ and the potential function $\psi$.  
Studying deeper such connections between existing notions might lead to interesting results, but is left for further work.

\subsection{Other perspectives}\label{perspectives}

Perspectives are both theoretical and practical. 
Statistical tests associated with our real-valued measures would be of great interest, to automatically detect when to use the vector-valued ones. 
Also, these risk measures appeal for being used in practical settings. 
 
Extending our definitions to extreme quantile levels and building related procedures is another line of study.
For instance, \cite{de2018stability}[Theorem 4.3] ensures the existence of a MK quantile map which is stable as moving further to the tail contours, which is not the case with the spherical uniform $U_d$. 
Our proposed center-outward superquantiles can be defined with other spherical reference measures, thus one can also calibrate these superquantiles as the work of \cite{de2018stability} suggests. 

In addition, the many applications of the univariate superquantile function in other fields than risk measurement, such as superquantile regression or optimization with a superquantile loss, appeal for further work, even more so with Theorems \ref{expshort_metrizes_weakcvgce} and \ref{cvgceSk}. 
Finally, the convergence of empirical entropic optimal transport towards its population counterpart is an active field, \cite{Barrio2022,gonzalez2022,Mena2019,pooladian2021entropic}, and it remains to study to what extent existing results adapt to the convergence of $\widehat{S}_{\ee,n}$ and $\widehat{E}_{\ee,n}$.

\bibliographystyle{abbrv} 
\bibliography{superquantiles.bib}


\end{document}